\documentclass[anglais]{cedram-aif}

\usepackage[english,francais]{babel}
\usepackage[latin1]{inputenc}
 \usepackage[T1]{fontenc}
 \usepackage{lmodern}

\usepackage{amsmath,amssymb, amsthm, latexsym}

\usepackage{graphicx}
\usepackage{mathrsfs}
\usepackage{color}
\usepackage{hyperref}
\title[Commuting vector fields on three manifolds]{Existence of common zeros for commuting vector fields on three manifolds}

\author{\firstname{Christian} \lastname{Bonatti}}

\address{Institut de Math\'ematiques de Bourgogne,\\
UMR 5584 du CNRS, Universit\'e de Bourgogne 
9, Avenue Alain Savary\\
21000 (Dijon)}

\email{bonatti@u-bougogne.fr}

\author{\firstname{Bruno} \lastname{Santiago}}
\address{Institut de Math\'ematiques de Bourgogne,\\
UMR 5584 du CNRS, Universit\'e de Bourgogne 
9, Avenue Alain Savary\\
21000 (Dijon)}

\email{bruno.rodrigues-santiago@u-bourgogne.fr}

\thanks{We thank Sebastien Alvarez for kind discussions helping to clean the arguments and Johan Taflin
for his kind interest in our work. 
We also thank  Martin Vogel for many useful conversations.
This work  was supported by the Projeto Ci\^encia Sem Fronteiras {\it Din\^amicas n\~ao hiperbolicas: 
aspectos topol\'ogicos e erg\'odicos} (CAPES (Brazil)) and by FAPERJ(Brazil).
We thank the kind hospitality of Institut de Math\'ematiques de Bourgogne, UMR 5584 du CNRS, Dijon, France 
and Departamento de Matem\'atica, PUC-Rio de Janeiro, Brazil. The referee of the first version of this paper made a very careful reading and provided many comments improving the text; we thank 
him/her.}



\begin{document}

\newtheorem{maintheorem}{Theorem}

\renewcommand{\themaintheorem}{\Alph{maintheorem}}
\newtheorem{clai}{Claim}

\newcommand{\margem}[1]{\marginpar{{\scriptsize {#1}}}}
\newcommand{\notation}{\textbf{Notation.  }}
\newcommand{\field}[1]{\mathbb{#1}}
\newcommand{\real}{\field{R}}
\newcommand{\complex}{\field{C}}
\newcommand{\integer}{\field{Z}}
\renewcommand{\natural}{\field{N}}
\newcommand{\rational}{\field{Q}}
\newcommand{\Tor}{\field{T}}
\newcommand{\torus}{\field{T}}
\newcommand{\Circ}{\field{S}}
\newcommand{\Proj}{\field{P}}

\newcommand{\al} {\alpha}       \newcommand{\Al}{\Alpha}
\newcommand{\be} {\beta}        \newcommand{\Be}{\Beta}
\newcommand{\ga} {\gamma}    \newcommand{\Ga}{\Gamma}
\newcommand{\de} {\delta}       \newcommand{\De}{\Delta}
\newcommand{\ep} {\eps}
\newcommand{\eps}{\varepsilon}
\newcommand{\ze} {\zeta}
\newcommand{\te} {\theta}       \newcommand{\Te}{\Theta}
\newcommand{\vte}{\vartheta}
\newcommand{\iot}{\iota}
\newcommand{\ka} {\kappa}
\newcommand{\la} {\lambda}      \newcommand{\La}{\Lambda}
\newcommand{\vpi}{\varpi}
\newcommand{\vro}{\varrho}
\newcommand{\si} {\sigma}       \newcommand{\Si}{\Sigma}
\newcommand{\vsi}{\varsigma}
\newcommand{\ups}{\upsilon}     \newcommand{\Up}{\Upsilon}
\newcommand{\vfi}{\phi}
\newcommand{\om} {\omega}       \newcommand{\Om}{\Omega}

\newcommand{\Z}{\mathbb{Z}}
\newcommand{\N}{\mathbb{N}}
\newcommand{\R}{\mathbb{R}}
\newcommand{\supp}{\operatorname{supp}}
\newcommand{\diam}{\operatorname{Diam}}
\newcommand{\dist}{\operatorname{dist}}
\newcommand{\const}{\operatorname{const}}
\newcommand{\topp}{\operatorname{top}}
\newcommand{\Leb}{\operatorname{Leb}}
\newcommand{\var}{\operatorname{var}}
\newcommand{\ov}{\overline}
\newcommand{\interior}{\operatorname{int}}
\newcommand{\lip}{\operatorname{Lip}_1(X;\R)}
\newcommand{\zero}{\operatorname{Zero}}
\newcommand{\per}{\operatorname{Per}}
\newcommand{\col}{\operatorname{Col}}
\newcommand{\dom}{\operatorname{dom}}
\newcommand{\sing}{\operatorname{Sing}}
\newcommand{\ind}{\operatorname{Ind}}

\def\AA{{\mathbb A}} 
\def\BB{{\mathbb B}} 
\def\CC{{\mathbb C}} 
\def\DD{{\mathbb D}}
\def\EE{{\mathbb E}} 
\def\FF{{\mathbb F}} 
\def\GG{{\mathbb G}} 
\def\HH{{\mathbb H}}
\def\II{{\mathbb I}} 
\def\JJ{{\mathbb J}} 
\def\KK{{\mathbb K}} 
\def\LL{{\mathbb L}}
\def\MM{{\mathbb M}} 
\def\NN{{\mathbb N}} 
\def\OO{{\mathbb O}} 
\def\PP{{\mathbb P}}
\def\QQ{{\mathbb Q}} 
\def\RR{{\mathbb R}} 
\def\SS{{\mathbb S}} 
\def\TT{{\mathbb T}}
\def\UU{{\mathbb U}} 
\def\VV{{\mathbb V}} 
\def\WW{{\mathbb W}} 
\def\XX{{\mathbb X}}
\def\YY{{\mathbb Y}} 
\def\ZZ{{\mathbb Z}}

\newcommand{\ti}{\tilde }
\newcommand{\Ptop}{P_{\topp}}
\newcommand{\htop}{h_{\topp}}
\newcommand{\un}{\underbar}
\newcommand{\cI}{{\mathcal I}}
\newcommand{\mL}{{\mathcal L}}
\newcommand{\mM}{{\mathcal M}}
\newcommand{\cH}{{\mathcal H}}
\newcommand{\cC}{\mathcal{C}}
\newcommand{\cK}{\mathcal{K}}
\newcommand{\cR}{\mathcal{R}}
\newcommand{\cP}{\mathcal{P}}
\newcommand{\cL}{\mathcal{L}}
\newcommand{\cE}{\mathcal{E}}
\newcommand{\cM}{\mathcal{M}}
\newcommand{\cB}{\mathcal{B}}
\newcommand{\cO}{\mathcal{O}}
\newcommand{\cG}{\mathcal{G}}
\newcommand{\cN}{\mathcal{N}}
\newcommand{\cQ}{\mathcal{Q}}
\newcommand{\cF}{\mathcal{F}}
\newcommand{\cT}{\mathcal{T}}
\newcommand{\cW}{\mathcal{W}}
\newcommand{\cU}{\mathcal{U}}
\newcommand{\cS}{\mathcal{S}}
\newcommand{\cA}{\mathcal{A}}
\newcommand{\cZ}{\mathcal{Z}}
\newcommand{\cX}{\mathcal{X}}
\newcommand{\tmu}{\tilde\mu}
\newcommand{\NB}{\margem{$*$}}
\newcommand{\ce}{{\mathcal E}}
\newcommand{\p}{\field{P}}

\selectlanguage{english}

\begin{abstract}In $1964$ E. Lima proved that commuting vector fields on surfaces 
with non-zero Euler characteristic have common zeros. Such statement is empty in 
dimension $3$, since all the Euler characteristics
vanish. Nevertheless, C. Bonatti proposed in $1992$ a local version, replacing the 
Euler characteristic by
the Poincar\'e-Hopf index of a vector field $X$ in a region $U$, denoted by 
$\operatorname{Ind}(X,U)$; he asked:

\emph{Given commuting vector fields $X,Y$ and a region $U$
where  $$\operatorname{Ind}(X,U)\neq 0$$ does
$U$ contain a common zero of $X$ and $Y$?} 

A positive answer was given in the case 
where $X$ and $Y$ are real analytic, in the same article where the above question was posed.  

In this paper, we prove the existence of common zeros for commuting $C^1$ vector fields $X$, 
$Y$ on a $3$-manifold,  in any region $U$ such that $\operatorname{Ind}(X,U)\neq 0$, 
assuming that the set of collinearity 
of $X$ and $Y$ is contained in a smooth surface. 
This is a strong indication that the results for analytic vector fields should
hold in the $C^1$ setting.
\end{abstract} 

\selectlanguage{francais}

\begin{altabstract}
En 1964 E. Lima a montré que les champs de vecteurs que commute dans une surface ont un zéro commun. Cette énonce est vide en dimension 3 puisque toutes les caractéristiques d'Euler sont nulles dans ce cas-là. Cependant, C. Bonatti a proposé 1992 une version locale, en remplaçant la caractéristique d'Euler par l'indice de Poincaré-Hopf d'un champ de vecteur $X$ dans une région $U$, qu'on denote par $\operatorname{Ind}(X,U)$. Il a proposé la question suivante: 
\emph{Étant donné deux champs de vecteurs commutant $X$ et $Y$ et une région compacte $U$ sur lequel  $$\operatorname{Ind}(X,U)\neq 0,$$ est-ce que $U$ contient un zéro commun de $X$ et $Y$?}

Une réponse positive a été donné dans le cas où $X$ et $Y$ sont réelle analytique, dans le même papier où la question au-dessus a été posé.    

Dans cet article on montre existence de zéros communs pour les champs de vecteurs de classe $C^1$ que commute en diménsion 3 pour chaque région $U$ telle que l'indice $\operatorname{Ind}(X,U)$ est non nul et en supposent en plus que le lieu de colinéarité entre $X$ et $Y$ est contenu dans une surface lisse. C'est une forte indication que le résultat pour les champs de vecteurs analytiques doit être vrai en régularité $C^1$. 
\end{altabstract}

\selectlanguage{english}

\maketitle

\section{Introduction}

One of the fundamental problems in dynamical systems is whether a given system possesses fixed points or not. 
A simple scenario to pose this question
is for the $\mathbb{Z}$-action generated by a diffeomorphism or  a homeomorphism of a manifold, or for the
continuous time dynamical system generated by the flow of a vector field. 
In both cases, the theories of Poincar\'e-Hopf
and Lefschetz indices
relate the topology of the 
ambient manifold with the existence of fixed points. 

Nonetheless, if one consider two commuting diffeomorphisms or two commuting vector fields 
i.e. vector fields $X$ and $Y$ whose flows satisfy:
\footnote{This definition can be adapted for non-complete vector fields, see Section \ref{s.notacao}} 
$$X_t\circ Y_s=Y_s\circ X_t,\:\:\forall(s,t)\in\mathbb{R}^2,$$
the existence of a fixed point for the action they generate  is  a wide  open question 
in dimensions $\geq 3$. 

The first result on this question is given by  the works on surfaces of 
Lima \cite{Lima_esfera}, \cite{Lima_geral}. He proves 
that  any family of commuting vector fields on a surface with non-zero Euler characteristic 
have a common zero. 
In the late eighties, \cite{Bonatti_esfera}  proved that 
commuting diffeomorphisms of the sphere $\SS^2$ which are  $C^1$-close to the identity have a common fixed point.
Later 
\cite{Bonatti_geral} extended this result to any surface with non-zero Euler characteristic 
(see other generalizations in \cite{DFF}\cite{Firmo}).  Then,
Handel \cite{Handel}   provided a topological invariant in $\ZZ/2\ZZ$ for a pair of 
commuting diffeomorphisms of the sphere  $\SS^2$ whose vanishing guarantees a common fixed point. 
This was further generalized
by Franks, Handel and Parwani \cite{FHP} for any number of commuting diffeomorphisms on the sphere 
(see \cite{Hirsch} and \cite{FHP2} for generalizations on other surfaces). 

It is worth to note, however, that two commuting 
\emph{continuous interval maps} may fail to have a common fixed point: 
an example is constructed in \cite{Boyce} of two continuous commuting, non-injective, maps 
of the closed interval which do not have a common fixed point.

In higher dimensions much less is known: 
\begin{itemize}
 \item One knows some  relation between the 
topology of the manifold and the dimension of the orbits of  
$\R^p$-actions (see \cite{MT1,MT2}). The techniques introduced in these works make
possible a 
simple proof of Lima's result, for smooth 
vector fields \cite{Turiel};
\item \cite{Bonatti_analiticos}  proved that  two commuting real analytic vector fields on an
 analytic $4$-manifold with non-zero Euler characteristic have a common zero. 
 The same statement does not make sense in dimension
three since every $3$-manifold has zero Euler characteristic. Nevertheless, a local result remains 
true in dimension three.
\end{itemize}

Before stating the result of \cite{Bonatti_analiticos} on $3$-manifolds, we briefly recall 
the notion of the Poincar\'e-Hopf index $\ind(X,U)$  of a vector field $X$ on a compact region $U$ 
whose boundary $\partial U$ is disjoint from the set $\zero(X)$. If $U$ is a small compact  
neighborhood of 
an isolated zero $p$
of the vector field $X$, then $\ind(X,U)$ is just
the classical  Poincar\'e-Hopf index $\ind(X,p)$ of $X$ at $p$. 
For a general compact region $U$ with $\partial U\cap\zero(X)=\emptyset$, 
one considers a small 
perturbation $Y$ of $X$  with only  finitely many 
isolated zeros  in $U$.  Then, we define the index $\ind(X,U)$ as the sum of the Poincar\'e-Hopf 
indices $\ind(Y,p)$, $p\in\zero(Y)\cap U$. 
We refer the reader to Section \ref{s.notacao} for details (in particular for the fact that $\ind(X,U)$ 
does not depend on the perturbation $Y$ of $X$).

Then, the main theorem of \cite{Bonatti_analiticos} says that every pair $X,Y$ of analytic commuting vector fields have a common zero in any compact region
$U$ such that $\ind(X,U)\neq 0$. By contraposition, if $X$ and $Y$ are analytic and commute but do not have common zeros then $\ind(X,U)=0$. By reducing
the compact region $U$ so that it separates $\zero(X)\cap U$ from $\zero(Y)$ one obtains the following statement: 


\begin{theo}[Bonatti \cite{Bonatti_analiticos}]
Let $M$ be a real analytic
$3$-manifold and $X$ and $Y$ be two analytic 
commuting vector fields over $M$. Let $U$ be a compact subset  of $M$ 
such that $\zero(Y)\cap U=\zero(X)\cap {\partial U}=\emptyset$. 
Then, 
$$\ind(X,U)=0.$$
\end{theo}

This statement is also true when $M$ has dimension $2$ and the vector fields are just $C^1$ 
(see Proposition 11 in \cite{Bonatti_geral}\footnote{The result stated in 
Proposition 11 of \cite{Bonatti_geral} is for $C^{\infty}$ vector fields, but 
the proof indicated there can be adapted for $C^1$ vector fields using cross-sections, 
in a similar way we do here in Section~\ref{s.reduzindo.a.prova}}). 
This motivates the following

\begin{conj}
\label{conjectura.local}
 Let $X$ and $Y$ be two $C^1$ commuting 
vector fields on a $3$-manifold $M$. Let $U$ be a  compact subset of $M$ such that 
$\zero(Y)\cap U=\zero(X)\cap {\partial U}=\emptyset$. 

Then, $\ind(X,U)=0$.
\end{conj}

This conjecture was stated as a problem in \cite{Bonatti_analiticos}. 

The goal of the present paper 
is to solve,  in the $C^1$-setting,  what was the main difficulty in the analytic case in \cite{Bonatti_analiticos}.
We explain now what was this difficulty in \cite{Bonatti_analiticos}.
A crucial role is played by the set of points of $U$ in which $X$ and $Y$ are collinear: 
$$\col(X,Y,U):=\{p\in U;\dim\left(\langle X(p),Y(p)\rangle\right)\leq 1\}.$$ 

In \cite{Bonatti_analiticos} the assumption that the commuting 
vector fields are analytic is used to say that $\col(X,Y,U)$ is either equals 
to $U$ or is an analytic
set of dimension at most 2. The case where $\col(X,Y,U)=U$ admits a direct proof.  In the other case, 
 a simple argument allows to assume that $\col(X,Y,U)$ is a surface. 
The main difficulty in \cite{Bonatti_analiticos} consists in proving that, if $\col(X,Y,U)$
is a smooth surface and $X$ and $Y$ are analytic, then the index of $X$ vanishes on $U$.

Our result is the following

\begin{maintheorem}
\label{teoprincipal}
Let $M$ be a $3$-manifold and $X$ and $Y$ be two $C^1$ commuting 
vector fields over $M$. Let $U$ be a compact subset of $M$ such that 
$\zero(Y)\cap U=\zero(X)\cap {\partial U}=\emptyset$. Assume that $\col(X,Y,U)$
is contained in a $C^1$-surface which is a compact and boundaryless submanifold of $M$. Then, 

$$\ind(X,U)=0.$$ 
\end{maintheorem}

The hypothesis ``\emph{ $\col(X,Y,U)$
is contained in a $C^1$-surface}'' consists in considering the simplest geometric configuration of $\col(X,Y,U)$
for which the conjecture is not trivial: if $(X,Y)$ is a counter example to the conjecture, then $\col(X,Y)$ cannot
be ``smaller'' than a surface.  More precisely, if $\ind(X,U)\neq 0$,
and if $Y$ commutes with $X$ then
the sets $\zero(X-tY)$ for small $t$ 
are not empty 
compact subsets of $\col(X,Y,U)$, invariant by the flow of $Y$ and therefore consist in orbits of $Y$. 
If $X$ and $Y$ are assumed without 
common zeros, every set $\zero(X-tY)$ consists on regular orbits of $Y$, thus is a $1$-dimensional lamination. 
Furthermore, these laminations are pairwise disjoint and vary semi-continously with $t$. In particular, $\col(X,Y,U)$ cannot
be contained in a $1$-dimensional submanifold of $M$.

Another (too) simple configuration would be the case where 
$X$ and $Y$ are everywhere collinear. This case has been treated in \cite{Bonatti_analiticos} and the same proof
holds at least in the $C^2$ setting. 

We believe that the techniques that we introduce here will be usefull to prove the conjecture, at least for 
$C^2$ vector fields.
\vskip 2mm

The proof of Theorem~\ref{teoprincipal} is by contradiction. 
The intuitive idea which guides the argument is that, 
at one hand,  the vector field $X$ 
needs to turn in all directions
in a non-trivial way in order to have a non-zero index. On the other hand, $X$ commutes with $Y$  
and therefore is 
invariant under the tangent 
flow of a non-zero vector field. 
The combination of this two properties  will lead to a contradiction.

This paper is organized as follows.
\begin{itemize}
\item In Section~\ref{s.resumo} we give an informal presentation of the proof, describing the main geometrical ideas.
\item In Section \ref{s.notacao} we give detailed definitions and state some classical facts that we shall use. 
\item In Section
\ref{s.reduzindo.a.prova} we  reduce the proof of Theorem \ref{teoprincipal}
to the proof of a slightly more technical version of it (see Lemma~\ref{l.prepared}), 
for which $U$ is a solid torus and $\col(X,Y,U)$
is a annulus foliated by periodic orbits of $Y$, and cutting $U$ in two connected components $U^+$ and $U^-$. 
\item In Section~\ref{Coordinates} we consider the projection $N$,  of 
the vector field
$X$ parallel to $Y$ on the normal bundle of $Y$. We show that $\ind(X,U)$ is 
related with the angular variations $\ell^+$ and $\ell^-$  of $N$ along generators of the fundamental 
group of each connected components $U^+$ and $U^-$  of 
$U\setminus \col(X,Y,U)$.  More precisely we will show in Proposition~\ref{p.link} that 
$$|\ind(X,U)|=|\ell^+-\ell^-|.$$

Assuming that at least one of $\ell^+$ and $\ell^-$ does not vanish, 
and the fact that $\col(X,Y,U)$ is a $C^1$-surface, we deduce in Proposition~\ref{p.derivada} that the 
 the first return map $\cP$ of $Y$ on a transversal $\Si_0$ is $C^1$-close to identity in a small neighborhood of $\col(X,Y,U)\cap \Si_0$.
\item In Section~\ref{s.proof}, still assuming that at least one of $\ell^+$ and $\ell^-$ does not vanish, 
we give a description of the dynamics of the 
first return map $\cP$.  If for instance $\ell^+\neq 0$ then every point in $\Si_0\cap U^+$ belongs to the stable set of a fixed 
point of $\cP$ (Lemma~\ref{l.stableset}). We will then use the invariance of these stable sets under the orbits of the normal vector field $N$ for 
getting a contradiction. 

\end{itemize}

We end this introduction by a general comment.  The accumulation of results proving 
the existence of common fixed points for commuting dynamical systems seems to indicate  
the possibility of a general phenomenon. However, our approach in 
Poincar\'e-Bendixson spirit has a difficulty which increases
drastically with the ambient dimension. We hope that this results will motivate other
attempts to study this phenomenon. 

\section{Idea of the proof}\label{s.resumo}

The proof is by contradiction.  We assume that there exists a \emph{counter example to the Theorem}, that is a pair of 
commuting $C^1$-vector fields $X,Y$ on a $3$-manifold $M$, and  a compact set $U$ so that 
\begin{itemize}
 \item the colinearity locus of $X$ and $Y$ in $U$ is contained in a closed $C^1$ surface $S$ in $M$. 
 \item $X$ is non-vanishing on the boundary $\partial U$ and the index  $\ind(X,U)$ is non-zero,  
 \item $X$ and $Y$ have no comon zero in $U$. 
\end{itemize}

\paragraph{Simplifying the counter examples}
A first step of the proof (see the whole Section~\ref{s.reduzindo.a.prova} and more specifically Lemmas~\ref{l.annulus} and \ref{l.prepared})  
consists in showing that, up to shrink the compact set $U$, and up to replace
the vector fields $X$ and $Y$ by (constant) linear combinations of $X$ and $Y$, one may assume futher, 
without loss of generality, that

\begin{itemize} 
\item the manifold $M$ is orientable;
\item $U$ is a solid torus $\DD^2\times (\RR/\ZZ)$;
 \item the surface $S$ is an annulus invariant by the flow of $X$ and $Y$ whose intersection with the boundary $\partial U$ is
 precisely its own boundary $\partial S$; 
 \item $Y$ is non-vanishing on $U$ and transverse to every factor $\Si_t=\DD^2\times\{t\}$, $t\in\RR/\ZZ$;
 \item for every $x\in S$ its $Y$ orbit is a periodic orbit contained in $S$, of period $\tau(x)$.
 \item the map $x\in S\mapsto  \tau(x)>0$ is of class $C^1$, constant on the $Y$-orbits, and its derivative is non vanishing on $S$; 
 \item as $X$ is colinear on $S$ to the non-vanishing vector field $Y$ one may write $X(x)=\mu(x)Y(x)$ for $x\in S$; 
 then the map $\mu$ is of class $C^1$, constant on the $Y$-orbits, 
 and  its derivative is non vanishing on $S$. 
\end{itemize}
We endow the solid torus $U$ with a basis $\cB(x)= (e_1(x), e_2(x), e_3(x))$ of $T_xM$ dependind continuously with $x$ and so that  $e_3(x)=Y(x)$,
the plane $<e_1(x),e_2(x)>$ is tangent to the disc $\Si_t$ through $x$, and for $x\in S$ the vector $e_1(x)$ is tangent to $S$. 

\begin{figure}[h]
\centering
\includegraphics[width=260pt,height=150pt]{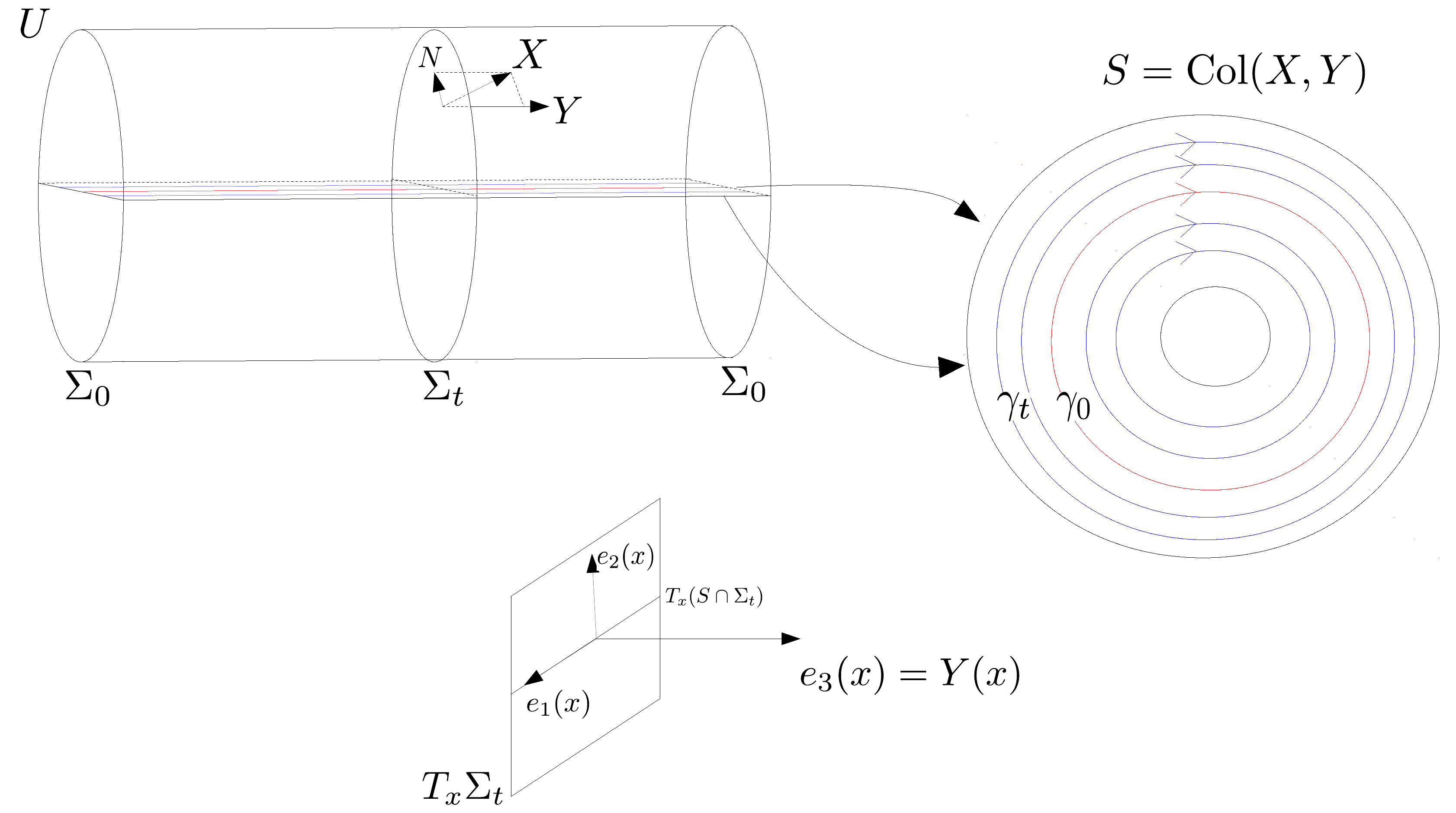}
\caption{Geometric configuration of a \emph{prepared counter example} and  the basis $\cB(x)=(e_1(x),e_2(x), e_3(x))$ at a point $x\in S\cap\Sigma_t$.}
\label{f.preparados}
\end{figure}

The pair of vector fields $(X,Y)$ with all these extra properties and endowed with the basis $\cB$ is called a \emph{prepared counter example}. More formally 
Lemmas~\ref{l.annulus} and \ref{l.prepared} announce that the existence of a counter example to Theorem~\ref{teoprincipal} implies the existence of a prepared counter example. The rest
of the proof consists in getting a contradiction from the existence of a prepared counter example. 

\paragraph{Calculating the index $\ind(X,U)$}
As $Y$ is transverse to the discs $\Si_t$ one can write the vector field $X$ as 
$$ X(x)= N(x) +\mu(x) Y(x)$$
where $N(x)$ is a vector tangent to the discs $\Si_t$, so that one may write $N(x)=\alpha(x) e_1(x) +\beta(x) e_2(x)$. This choice of coordinates
allows us to consider the vector $N(x)$ as a vector on the plane $\RR^2$. Notice that $N(x)$ vanishes if and only if $x\in S$.

Thus, given any closed curve $\gamma\subset U\setminus S$ which is freely homotopic to $\{(0,0)\}\times (\RR/\ZZ)$ in $U$, one defines the 
\emph{linking number of $N$ along $\gamma$} as being the number of turns  given by $N(x)$ 
(considered as a non-vanishing vector on $\RR^2$) as $x$ runs along $\gamma$. One easily checks that this 
linking number only depends on the connected component of $U\setminus S$ which contains $\gamma$. Since $U\setminus S$ 
consist in two connected components 
$U^+$ and $U^-$, there are only two linking numbers which we denote by $\ell^+$ and $\ell^-$ and we prove (Proposition~\ref{p.link}) 
$$|\ind(X,U)|=|\ell^+-\ell^-|.$$

\paragraph{The normally hyperbolic case}
By assumption the annulus $S$ is foliated by periodic orbits of $Y$ so that the derivative associated to these periodic orbits is $1$ in the direction of $S$.

Note that, if one of the periodic orbits of $Y$ in $S$ is partially hyperbolic, that is, has an eigenvalue different from $1$, then it admits a
neighborhood foliated by the local stable manifolds of the nearby periodic orbits. As the vector field $X$ commutes with $Y$ and preserves each periodic orbit in $S$, 
it preserves each leaf of this foliation.  Thus the normal vector $N$ is tangent to this foliation and one deduces that the linking numbers $\ell^+$ and $\ell^-$ both vanish. 
Thus $\ind(X,U)$ also vanishes leading to a contradiction (see Lemma~\ref{l.casoph} which formalizes this argument). 

\begin{figure}[h]
\centering
\includegraphics[width=300pt,height=100pt]{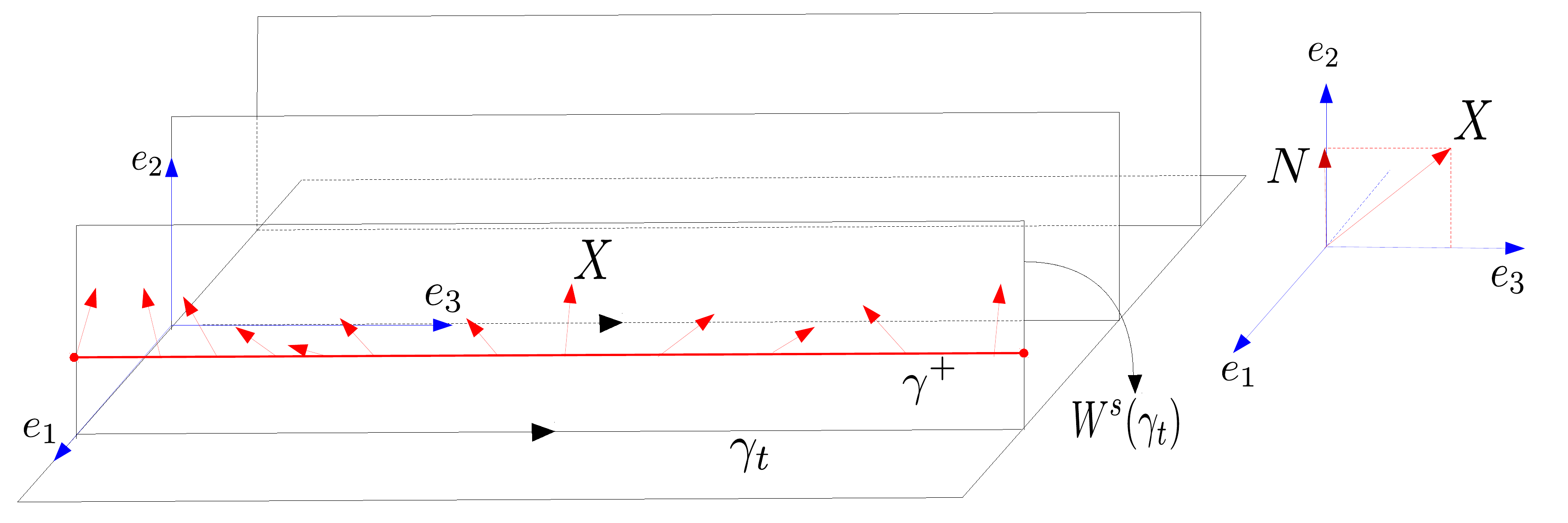}
\caption{As $X$ is tangent to the stable manifolds, it cannot turn in all directions. Its normal component
$N$ is everywhere collinear with $e_2$, so that both linking numbers vanish.}
\label{f.casoph}
\end{figure}

\paragraph{The first return map $\cP$, the derivative of the return time, and the angular variation of $N$}
The vector field $N$ does not commute with $Y$ but it is invariant under the holonomies of $Y$ of the cross sections $\Si_t$. One deduces that $N$ almost cannot turn along the 
orbits of $Y$.  This motivate us to calculate the linking number $\ell^\pm$ along particular closed curves obtained in the following way: we follow the $Y$-orbit of a point $x\in\Si_0\setminus S$
until its first return $\cP(x)$ then we close the curve by joining
a small geodesic segment in $\Si_0$. This allows us to prove that the angular variation of $N$ is larger than $2\pi$ along the segment in $\Si_0$ joining
$x$ to $\cP^2(x)$ (second return of $x$ in $\Si_0$)  (Corollary~\ref{c.gira}).

\begin{figure}[h]
\centering
\includegraphics[width=210pt,height=100pt]{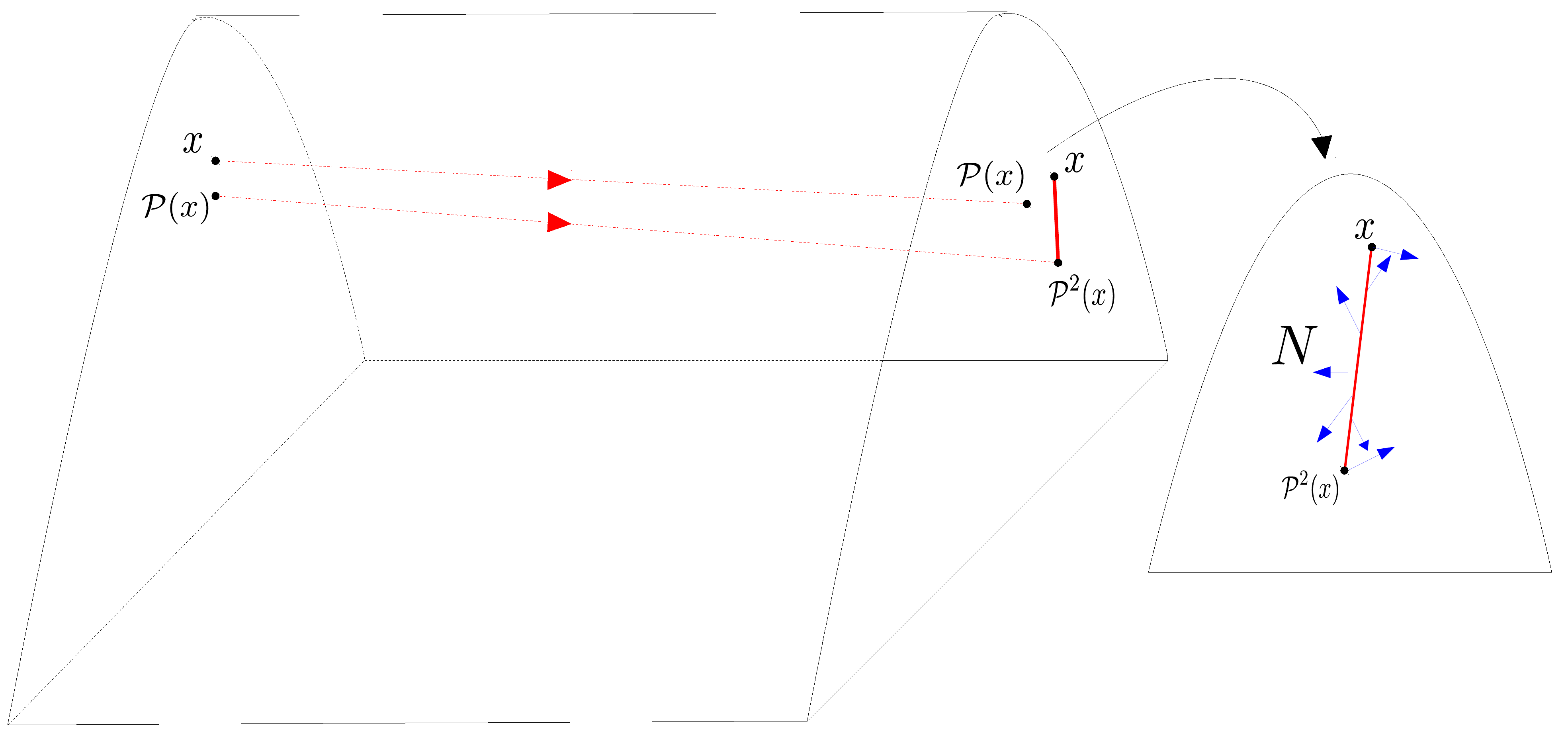}
\caption{$N$ has large angular variation along the segments $[x,\cP^2(x)]$.}
\label{f.ngira}
\end{figure}
 
Now the contradiction we are looking for will be found in a subttle analysis of the dynamics of the first return map $\cP$ together 
 with the dynamics of the vector $N$, which commute together. In particular Corollary~\ref{derivative.return.time} links  the direction of the vector field $N(x)$ (through the derivative in the direction of $N$ of the return time of
 $Y$ on $\Si_0$) with the variation $\mu(\cP(x))-\mu(x)$.  The big angular variation of $N$ along the small segments $[x,\cP^2(x)]$ lead us to prove that
 \begin{itemize}
  \item the derivative of $\cP$ is the identity at any point of $S$ (Proposition~\ref{p.derivada})
  \item $\cP$ almost preseves the levels of the map $\mu$ (see Lemmas~\ref{l.angle} and \ref{l.angle2}): one deduces that (up to exchange $\cP$ by $\cP^{-1}$) every orbit of $\cP$ converges to a point of $S$. 
 \end{itemize}
\paragraph{Invariant stable sets and the contradiction}
The second property above allows us to obtain stable sets for the points in $S\cap \Si_0$ (with respect to the first return map $\cP$) and to show that these stable sets are (as in the normally hyperbolic case) 
invariant under the vecter field $N$.  Unluckely the concluison is not so straightforward as in the normally hyperbolic case. We first prove that there is an $N$-orbit which is invariant uner $\cP$ (Lemma~\ref{l.invariant})
and we prove that the angular variation of $N$, between $x$ and $\cP^2(x)$,  along such a $\cP$ invariant $N$-orbit is arbitrarilly small when $x$ is close to $S\cap \Si_0$: that is the announced  contradiction.

\section{Notations and definitions}\label{s.notacao}
In this paper $M$ denotes a $3$-dimensional manifold. Whenever $X$ is a vector field over $M$, we  
denote $\zero(X)=\{x\in M;X(x)=0\}$ and $\zero(X,U)=\zero(X)\cap U$, for any subset $U\subset M$. 
We shall denote its flow by $X_t$.
A compact set $\Lambda\subset M$ is  \emph{invariant under the flow of  $X$} if 
$X_t(\Lambda)=\Lambda$ for every $t\in\R$. 

If $X$ and $Y$ are vector fields on $M$ we denote by $\col(X,Y)$ the set of points $p$ for which $X(p)$ 
and $Y(p)$ are collinear: 
$$\col(X,Y)=\{p\in M, \dim(\langle X(p),Y(p))\rangle\leq 1\}.$$
If $U\subset M$ is a compact region we denote $\col(X,Y,U)=\col(X,Y)\cap U$. 

\subsection{The Poincar\'e-Hopf index.}\label{s.index}
In this section we recall the classical definition and properties of the Poincar\'e-Hopf index. 

Let $X$ be a continuous vector field of a manifold $M$ of dimension $d$ and $x\in M$ 
be an isolated zero of $X$. The \emph{ Poincar\'e-Hopf index  $\operatorname{Ind}(X, x)$} is defined as 
follows: consider  local coordinates $\varphi\colon U\to\R^d$ 
defined in a neighbohood $U$ of $x$. Up to shrink $U$ one may assume that $x$ is the unique zero of $X$ 
in $U$.   Thus for $y\in U\setminus\{x\}$,  $X(y)$ expressed in that coordinates is a non vanishing vector 
of $\R^d$, and  $\frac 1{\|X(y)\|} X(y)$ 
is a unit vector
hence belongs to the sphere $\SS^{d-1}$.  Consider a small ball $B$ centered at $x$. 
The map $y\mapsto \frac 1{\|X(y)\|} X(y)$ induces a 
continuous map from 
the boundary $S=\partial B$ to $\SS^{d-1}$. The Poincar\'e-Hopf index  $\operatorname{Ind}(X, x)$ 
is the topological degree of this map. 

\begin{rema}\label{r.indice} The Poincar\'e-Hopf index  $\operatorname{Ind}(X, x)$ of an isolated zero $x$ does not depend
on a choice of a local orientation of the manifold at $x$.  

For instance, the Poincar\'e Hopf index of 
a hyperbolic zero $x$ is 
$$\operatorname{Ind}(X, x)= (-1)^{dim E^s(x)},$$
where $E^s(x)\subset T_xM$ is the stable space of $x$.

More conceptually, a change of the local orientation of $M$ at $x$: 
\begin{itemize}
 \item composes the map $x\mapsto \frac{X(x)}{\|X(x)\|}$ with a symmetry of the sphere $S^{d-1}$
 \item changes the orientation of the sphere $\partial B$, where $B$ is a small ball around $x$.
\end{itemize}
Therefore the topological degre of the induced map from $\partial B$ to $S^{d-1}$ is kept unchanged. 
\end{rema}

Assume now that $U\subset M$ is a compact subset and that $X$ does not vanish on the boundary $\partial U$. 
The   \emph{ Poincar\'e-Hopf index  $\operatorname{Ind}(X, U)$} is defined as follows: 
consider a small perturbation $Y$ of $X$ so 
that the set 
of zeros of $Y$ in $U$ is finite. A classical result asserts that  the sum of the indices of 
the zeros of $Y$ in $U$ does not depend on the perturbation $Y$ of $X$; 
this sum is the
Poincar\'e-Hopf index  $\operatorname{Ind}(X, U)$. Here, \emph{small perturbation} means that $Y$ is
homotopic to $X$ through vector fields 
which do not vanish on 
$\partial U$. More precisely
\begin{prop} If $\{X^t\}_{t\in[0,1]} $ is a continuous family of vector fields so that 
$\zero(X^t)\cap \partial U=\emptyset$, then $\operatorname{Ind}(X^t, U)$ does not depend on $t\in[0,1]$.  
\end{prop}

We say that a compact subset $K\subset \zero(X)$ is \emph{isolated} if there is a compact neighborhood $U$ of 
$K$ so that $K=\zero(X)\cap U$;  the neighborhood $U$ is called an \emph{isolating neighborhood of $K$}. 
The index  $\operatorname{Ind}(X, U)$ does not depend of the isolating neighborhood  $V$  of $K$. 
Thus $\operatorname{Ind}(X, U)$ is called \emph{the index of $K$} 
and denoted  $\operatorname{Ind}(X, K)$.

\subsection{Calculating the Poincar\'e Hopf index}
Next Lemma~\ref{l.Using the tubular neighboorhod to calculate the index} provides a practical method for calulating the index of 
a vector field $X$ in some region where it may have infinitely many zeros, without performing perturbations of $X$. 

Assume now that $\partial U$ is a codimension one submanifold and that $U$ is endowed 
with $d$ continous vector fields $X^1\dots X^d$ so that, at every point 
$z\in U$, $(X^1(z),\dots,X^d(z))$ is a basis 
of the tangent space $T_z M$. 
Once again, one can express the vector field $X$ in this 
basis so that the vector $X(y)$, for  $y\in U$, can be considered as a vector of $\R^d$. 
One defines in such a way  a map $g\colon \partial U\to \SS^{d-1}$ by 
$y\mapsto g(y)=\frac 1{\|X(y)\|} X(y)$.  

As $\partial U$ has dimension $d-1$, and is oriented as the boundary of $U$,  
this map has a topological degree.  
A classical result from homology theory implies the following

\begin{lemm}
\label{l.Using the tubular neighboorhod to calculate the index} With the notations above
the topological degre  of $g$ is  $\operatorname{Ind}(X, U)$.  
\end{lemm}

In particular it does not depend on the choice of the vector fields $X^1\dots X^d$.

\begin{rema}\label{r.Using the tubular neighboorhod to calculate the index} By Lemma~\ref{l.Using the tubular neighboorhod to calculate the index}, 
if there is $j\in\{1,\dots,d\}$ so that  the vector $X(x)$ is not colinear to  $X^j(x)$, for every
$x\in\partial U$, then  $\operatorname{Ind}(X, U)=0$. 
 
\end{rema}

\subsection{Topological degree of a map from $\TT^2$ to $\SS^2$}
We consider the sphere $\SS^2$ (unit sphere of $\RR^3$) endowed 
with the north and south poles denoted $N=(0,0,1)$ and $S=(0,0,-1)$ respectively.

We denote by $\SS^1\subset \SS^2$ the equator, oriented as the unit circle of $\RR^2\times\{0\}$. 
For $p=(x,y,z)\in \SS^2\setminus\{N,S\}$ we call \emph{projection of $p$ on $\SS^1$ along the meridians} 
the point $\frac1{\sqrt{x^2+y^2}}(x,y,0)$, which is  intersection 
of $\SS^1$ with the unique half meridian containing $p$.

\begin{lemm}\label{l.sphere}
 Let $\Phi\colon \SS^2\to \SS^2$ be a continuous map so 
 that $\Phi^{-1}(N)=\{N\}$ and $\Phi^{-1}(S)=\{S\}$. 
 
 Let $\varphi\colon\SS^1\to \SS^1$ be defined as follows: the point $\varphi(p)$, for $p\in\SS^1$,  is the projection of 
 $\Phi(p)\in \SS^2\setminus\{N,S\}$ on $\SS^1$ along the meridians of $\SS^2$. 
 
 Then the topopological degrees of $\Phi$ and $\varphi$ are equal. 
 
\end{lemm}
As a direct consequence one gets
\begin{coro}\label{c.sphere}
 Let $\Phi\colon \SS^2\to \SS^2$ be a continuous map so 
 that $\Phi^{-1}(N)=\{S\}$ and $\Phi^{-1}(S)=\{N\}$. 
 
 Let $\varphi\colon\SS^1\to \SS^1$ be defined as follows: the point  $\varphi(p)$, for $p\in\SS^1$, is the projection of 
 $\Phi(p)\in \SS^2\setminus\{N,S\}$ on $\SS^1$ along the meridians of $\SS^2$. 
 
 Then the topopological degrees of $\Phi$ and $\varphi$ are opposite. 
 
\end{coro}

We consider now the torus $\TT^2=\RR/\ZZ\times \RR/\ZZ$.  As a direct consequence of 
Lemma~\ref{l.sphere} and Corollary~\ref{c.sphere} one gets:

\begin{coro}\label{c.torus}
Let $\Phi\colon \TT^2\to \SS^2$ be a continuous map so 
 that $\Phi^{-1}(N)=\{0\}\times \RR/\ZZ$ and $\Phi^{-1}(S)=\{\frac12\}\times \RR/\ZZ$. 
 
 Let $\varphi_+\colon\{\frac14\}\times \RR/\ZZ\to \SS^1$ 
 (resp. $\varphi_-\colon\{\frac34\}\times \RR/\ZZ\to \SS^1$) be defined as follows:
 the point $\varphi_+(p)$ (resp. $\varphi_-(p)$)
 is the projection of 
 $\Phi(p)\in \SS^2\setminus\{N,S\}$ on $\SS^1$ along the meridians of $\SS^2$. 
 
 Then $$ |\deg(\Phi)|=|\deg(\varphi_+)-\deg(\varphi_-)|$$
 where $\deg()$ denotes the topological degree, and $\{\frac14\}\times \RR/\ZZ$ and 
 $\{\frac34\}\times \RR/\ZZ$ are endowed with the positive orientation of $\RR/\ZZ$. 
 
\end{coro}
\begin{proof}
Indeed, $\Phi$ is homotopic (by an homotopy preserving $\Phi^{-1}(N)$ and $\Phi^{-1}(S)$) to the map 
$\Phi_{d^+,d^-}\colon\RR/\ZZ\times \RR/\ZZ\to \SS^2$ defined as follows
\begin{itemize}
 \item $\Phi_{d^+,d^-}(s,t)= \left(|\sin(2\pi s)|\cdot e^{2i\pi d^+t},\cos(2\pi s)\right)$ if $s\in[0,\frac 12]$,
 \item $\Phi_{d^+,d^-}(s,t)= \left(|\sin(2\pi s)|\cdot e^{2i\pi d^-t},\cos(2\pi s)\right)$ if $s\in[0,\frac 12]$.
\end{itemize}
where $d^+$ and $d^-$ are $\deg(\varphi_+)$ and $\deg(\varphi_-)$, respectively.
\end{proof}

\subsection{Commuting vector fields: local version} 

There are two usual definitions for commuting vector fields: 
we can require that the flows of $X$ and $Y$ commute; one may also require that the 
Lie bracket $[X,Y]$ vanishes.
These two definitions coincide for $C^1$ vector fields on compact manifolds.
On non compact manifolds we just have the commutation of the flows for small times, 
as explained more precisely below.

Let $M$ be a (not necessarily compact) manifold and $X$, $Y$ be $C^1$-vector fields on $M$.  
The Cauchy-Lipschitz theorem asserts that the flow of $X$ and $Y$ are locally defined but they may not be 
complete. 
  
We say that \emph{$X$ and $Y$ commute} if for every point $x$ there is 
$t(x)>0$ so that for every $s,t\in[-t(x),t(x)]$  the compositions $X_t\circ Y_s(x)$ 
and $Y_s\circ X_t(x)$ are defined 
and coincide. Thus, the local diffeomorphism $X_t$ carries integral curves of $Y$
into integral curves of $Y$, and vice-versa. 

\begin{rema}There are (non complete) commuting vector fields, a point $x$ and $t>0$ and $s>0$  so that  both $X_t\circ Y_s(x)$ 
and $Y_s\circ X_t(x)$ are defined  but are different. Let us present an example. 

Consider $\CC^* = \RR^2\setminus \{(0,0)\}$ endowed with the two vector fields $\tilde X=\frac\partial{\partial x}$ and $\tilde Y=\frac\partial{\partial y}$. 
Note that $\tilde X$ and $\tilde Y$ commute. 
Consider the $2$-cover $\phi\colon \CC^*\to \CC^*$ defined by $z\mapsto z^2$.  Let $X$ and $Y$ the lifts for $\phi$ of $\tilde X$ and $\tilde Y$, respectively. 
Then $X$ and $Y$ commute.  However, consider the point 
$\tilde x=(-\frac 1{\sqrt{2}}, -\frac 1{\sqrt{2}})= e^{i\pi \frac54}$. Consider $x=e^{i\pi \frac58}$, so that $\phi(x)=\tilde x$.

Next figure illustates the fact that $X_t Y_t(x)$ and $Y_tX_t(x)$ are well defined but distinct, for $t= \sqrt{2}$: 
\begin{figure}[h]
\centering
\includegraphics[width=330pt,height=115pt]{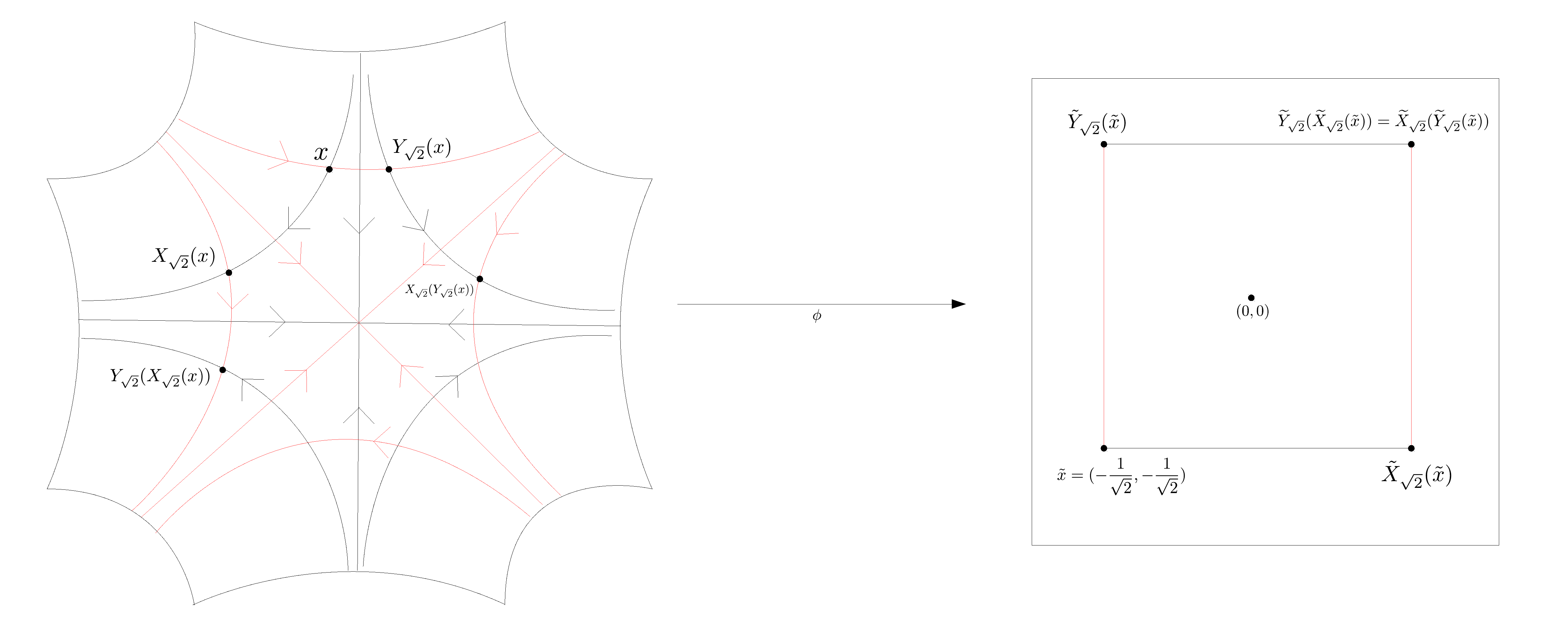}
\caption{The flow $X$ commutes with the flow of $Y$ until one of the composed orbits $X_tY_s$ crosses one of the axes.}
\label{f.localcomute}
\end{figure}

\end{rema}

 Next section states straightforward consequences of this definition.

\subsection{Commuting vector fields: first properties}\label{ss.commuting-first}
If $X$ and $Y$ are commuting vector fields then:
\begin{enumerate}
\item  for every $a,b,c,d\in\RR$, $aX+bY$ commutes with $cX+dY$;
\item  for every $x\in M$, and any $t\in \RR$ for which $Y_t$ is defined, one has 
 $$DY_t(x)X(x)=X(Y_t(x))$$
\item if $x\in\zero(X)$, 
 Then for any $t\in \RR$ for which $Y_t$ is defined, $Y_t(x)\in\zero(X)$.
 \item 
 $\col(X,Y)$ is invariant under the flow of $X$ in the following meaning: 
 if $x\in \col(X,Y)$ and if $X_t(x)$ is defined, then $X_t(x)\in\col(X,Y)$; in the same way, 
 $\col(X,Y)$ is invariant under the flow of $aX+bY$ for any $a,b\in\RR$. 
 \item
\label{d.ratio}
 if $\zero(Y)=\emptyset$, Then for each point $x$ in $\col(X,Y)$ there is $\mu(x)\in \RR$ 
 so that $X(x)=\mu(x)Y(x)$. The map $x\mapsto\mu(x)$ is called the \emph{ratio 
 between $X$ and $Y$ at $x$} and is continuous on $\col(X,Y)$ and can be
 extended in a $C^1$ map on the ambient manifold. 
\begin{proof}
One can extend $\mu$ to a small neighboorhod $V$ of $\col(X,Y)$ in the following way: given some riemannien metric on $M$, if $V$ is small enough then
$X$ is never orthogonal to $X$ in $V\setminus\col(X,Y)$. Thus, one can consider the vector field $Z$ obtained as the orthogonal projection of $X$ onto the direction of $Y$.
It is clear that $Z|_{\col(X,Y)}=X$ and that there exists a $C^1$ function $\psi:V\to\RR$ such that $Z(x)=\psi(x)Y(x)$. Since $\zero(Y)=\emptyset$, this implies that $\psi$ is an extension of $\mu$.
Moreover, clearly $\psi$ extends to $M$.\footnote{We shall see in Section~\ref{Coordinates} a particular case of this construction.}
\end{proof}
 
\item the ratio $\mu$ defined on $\col(X,Y)$ is invariant under the flow of $aX+bY$ for any $a,b\in \RR$.
\item if $\gamma$ is a periodic orbit of $X$ of period $\tau$ and if $t\in\RR$ is such that $Y_t$ 
is defined on 
 $\gamma$, then $Y_t(\gamma)$ is a periodic orbit of $X$ of period $\tau$. 
 The same occurs with the images by the flow of $cX+dY$ of periodic orbits of $aX+bY$, for $a,b,c,d\in \RR$.
 \item as a consequence of the previous item, if $\gamma$ is a periodic orbit of $X$ of period $\tau$ and if
 $\gamma$ is isolated among the periodic orbits of $X$ of the same period $\tau$, then $\gamma$ 
 is invariant under
 the flow of $Y$; as a consequence, $\gamma\subset \col(X,Y)$. 
\end{enumerate}

\subsection{Counter examples to Theorem~\ref{teoprincipal}}

Our proof is a long proof by \textit{reductio ad absurdum}. To achieve this goal, we shall first 
show that the existence of a pair $(X,Y)$ which do not satisfy the conclusion of
Theorem~\ref{teoprincipal} implies 
the existence of other pairs with simpler geometric behaviors. 

For this reason, it will be convenient to define the notion of counter examples in a formal manner.

\begin{defi}Let $M$ be a $3$-manifold, $U$ a compact subset of $M$ and $X$, $Y$ be $C^1$ vector fields on 
$M$.  We say that $(U,X,Y)$ is a \emph{counter example to Theorem~\ref{teoprincipal}} if
\begin{itemize}
 \item $X$ and $Y$ commute
 \item $\zero(Y)\cap U=\emptyset$
 \item $\zero(X)\cap \partial U=\emptyset$
 \item $\operatorname{Ind}(X, U)\neq 0.$
 \item the collinearity locus, $\col(X,Y,U)$, is contained in a $C^1$  surface which is a closed 
 submanifold of $M$.
\end{itemize}
\end{defi}

Let us illustrate our simplifying procedure by a simple argument:

\begin{rema}\label{r.orientable} If $M$ is a $3$-manifold carrying a counter example $$(U,X,Y)$$ to Theorem~\ref{teoprincipal}, 
then there is
an orientable manifold carrying a counter example to Theorem~\ref{teoprincipal}. 
Indeed, consider the orientation cover 
 $ \tilde M\to M$  and $\tilde U,\tilde X,\tilde Y$  the lifts of 
 $U,X,Y$ on $\tilde M$. Then the Poincar\'e-Hopf index of $\tilde X$ on $\tilde U$ is twice the one of $X$ 
 on $U$, and $(\tilde U,\tilde X,\tilde Y)$ is a 
 counter example to Theorem~\ref{teoprincipal}.
 
\end{rema}

Thus  we can assume (and we do it) without loss of generality that $M$ is orientable. 

Most of our simplifying strategy will now consist in combinations of the  following remarks

\begin{rema} If $(U,X,Y)$ is a counter example to Theorem~\ref{teoprincipal}, then there is $\varepsilon>0$ so
that  $(U, aX+bY,cX+dY)$ is also a counter example to Theorem~\ref{teoprincipal}, for every $a,b,c,d$ with 
$|a-1|<\varepsilon$, $|b|<\varepsilon$, $|c|<\varepsilon$ and $|d-1|<\varepsilon$. 
\end{rema}

\begin{rema} If $(U,X,Y)$ is a counter example to Theorem~\ref{teoprincipal}, then
  $(V, X,Y)$ is also a counter example to Theorem~\ref{teoprincipal} for any compact set $V\subset U$ 
  containing $\zero(X,U)$ in its interior.
\end{rema}

\begin{rema} Let $(U,X,Y)$ be a counter example to Theorem~\ref{teoprincipal} and assume that 
$\zero(X,U)= K_1\cup \dots \cup K_n$, where the $K_i$ are pairwise disjoint compact sets. 
Let $U_i\subset U$ be compact neighborhood of $K_i$ so that the $U_i$, $i=1,\dots, n$, are pairwise disjoint. 

Then there is $i\in\{1,\dots,n\}$ so that  $(U_i,X,Y)$ is  a counter example to Theorem~\ref{teoprincipal}.
\end{rema}

\section{Prepared counter examples $(U,X,Y)$ to Theorem~\ref{teoprincipal}}\label{s.reduzindo.a.prova}

\paragraph{Simplifying $\col(X,Y,U)$}
The aim of this paragraph is to prove
\begin{lemm}\label{l.annulus} If $(U,X,Y)$ is a 
counter example  to Theorem~\ref{teoprincipal} then there is a
 counter example $(\tilde U,\tilde X,\tilde Y)$ to Theorem~\ref{teoprincipal} and $\mu_0>0$ with the following property:

\begin{itemize}
\item for any $t\in [-\mu_0,\mu_0]$, the set of zeros of $\tilde X-t\tilde Y$ in $\tilde U$ consists 
precisely in $1$ periodic orbit $\gamma_t$ of $\tilde Y$;

\item for any $t\notin [-\mu_0,\mu_0]$, the set of zeros of $\tilde X-t\tilde Y$ in $\tilde U$ is empty; 
\item $\col(\tilde X,\tilde Y,\tilde U)$ is a $C^1$ annulus;
\item there is a $C^1$-diffeomorphism 
 $\varphi\colon\RR/\ZZ\times [-\mu_0,\mu_0]\to \col(\tilde X,\tilde Y,\tilde U)$ so that, for every 
 $t\in[-\mu_0,\mu_0]$, one has 
 $$\varphi(\RR/\ZZ\times\{t\})=\gamma_t.$$

\end{itemize}
\end{lemm}

\begin{proof}  By hypothesis $\col(X,Y,U)$ is contained in a $C^1$-surface $S$.  
Notice that there is $\mu_1>0$ so that for any $t\in[-\mu_1,\mu_1]$ one has 
$\zero(X-t Y)\cap\partial U=\emptyset$ and 
$\operatorname{Ind}(X-t Y, U)\neq 0$. In particular, we have that $\zero(X-tY,U)\neq\emptyset$.

As $X-tY$ and $Y$ commute, $\zero(X-t Y,U)$ is invariant under the flow of $Y$.  Futhermore,
as $\zero(X-t Y)$ does not intersect $\partial U$ the $Y$-orbit of a point $x\in \zero(X-t Y,U)$ 
remains in the compact set
$U$ hence is complete. 

Consider now the ratio function $\mu\colon \col(X,Y)\to\RR$, defined in the item (\ref{d.ratio}) of 
Subsection \ref{ss.commuting-first}. It follows that, for $x\in \col(X,Y)$,  
$\mu(x)=t\Leftrightarrow x\in\zero(X-tY)$. The map $\mu$ is invariant under the flows of 
$X$ and $Y$ (on $\col(X,Y)$). 
As mentioned in \ref{ss.commuting-first}, the  map $\mu$ 
can be extended on $M$ as  $C^1$ map still denoted $\mu$ 
(but no more $X,Y$-invariant). 
 
Let $\cL=\bigcup_{t\in[-\mu_1,\mu_1]}\zero(X-tY)\cap U$. Then $\cL$ is a compact set, contained in $S$ disjoint from
the boundary of $U$ and invariant under $Y$: it is a compact lamination of $S$. 
 
By applying the flox box theorem and a standard compactness argument we can take $\si\subset S$ a union of finitely many compact segments with end points out of $\cL$ and
so that the interior of  $\sigma$
cuts  transversely every orbit of $Y$ contained in $\cL$. 

\begin{clai}
Lebesgue almost every $t\in[-\mu_1,\mu_1]$ is a regular value of the restriction of $\mu$ to $\si$. 
\end{clai}
\begin{proof}Recall that Sard's theorem requires a regularity $n-m+1$ if one consider maps from an 
$m$-manifold to an $n$-manifold.  As $\mu$ is $C^1$ and $\dim\si=1$ we can apply Sard's theorem to the restriction of 
$\mu$ to $\sigma$, concluding. 
\end{proof}

Consider now a regular value $t\in(-\mu_1,\mu_1)$ of the restriction of $\mu$ to $\si$. Then 
$\mu^{-1}(t)\cap{\si}$ consists in finitely many points. Furthermore, $\mu^{-1}(t)\cap{\si}$ contains 
$\zero(X-tY)\cap \sigma$. 

\begin{clai}For $t\in [-\mu_1,\mu_1]$, regular value  of the restriction of $\mu$ to $\si$, 
the compact set $\zero(X-tY)\cap U$ consists in finitely many periodic orbits $\gamma_i$, $i\in\{1,\dots, n\}$ of $Y$. 
\end{clai}
\begin{proof} $\zero(X-tY)\cap U$ is a compact sub lamination of $\cL\subset S$ 
consisting of orbits of $Y$, and contained in 
$\mu^{-1}(t)$. Now,
$\si$ cuts transversely each orbit of this lamination and $\si\cap \mu^{-1}(t)$ is finite. One deduces that
$\zero(X-tY)\cap U$ consists in finitely many compact leaves, concluding.
\end{proof}

We now fix a regular value $t\in(-\mu_1,\mu_1)$ of the restriction of $\mu$ to $\si$.

Since $\zero(X-tY,U)\neq\emptyset$ we have that the integer $n$ of the above claim is positive.
Moreover, notice that $\ind(X-tY,U)=\sum_{i=1}^n\ind(X-tY,\gamma_i)$.  Thus there is $i$ so that 
$$\ind(X-tY,\gamma_i)\neq 0.$$
\begin{clai}
 There is a neighborhood $\Ga_i$ of $\gamma_i$ in $S$ which is contained in $\col(X,Y,U)$ and which consists in periodic 
 orbits of $Y$. 
\end{clai}
\begin{proof}
Let $p$ be a point in $\si\cap \gamma_i$. As $p$ is a regular point of the restriction of $\mu$ to $\sigma$
 there is a segment $I\subset \si$ centered at $p$ so that the restriction of $\mu$ to $I$ is injective and the 
 derivative of $\mu$ does not vanish. 
 
 As $\gamma_i$ has non-zero index for any $s$ close enough to $t$, $\zero(X-sY)$ contains an isolated  
 compact subset $K_s$
 contained in a small neighborhood of $\gamma_i$, and hence in $U$, 
 thus in $\col(X,Y,U)$ and thus in a small neighborhood
 of $\gamma_i$ in $\cL\subset S$.  This implies that each orbit of $Y$ contained in $K_s$ cuts $I$. 
 However, $\mu$ is constant equal to $s$ on $K_s$ and thus $\mu^{-1}(s)\cap I$ consist in a unique point. One deduces that 
 $K_s$ is a compact orbit of $Y$. 
 
 Since this holds for any $s$ close to $t$, one obtain that any point $q$ of $I$ close to $p$ is the 
 intersection point of $K_{\mu(q)}\cap I$. In other words, a neighborhood of $p$ in $I$
 is contained in $\col(X,Y,U)$ and the corresponding leaf of $\cL$ is a periodic orbit of $Y$, concluding.
\end{proof}

Notice that $\Ga_i$ is contained in  $\col(X,Y,U)$ so that the function $\mu$ is invariant under $Y$ on $\Ga_i$. As the 
derivative of $\mu$ is non vanishing (by construction) on $\Ga_i\cap I$ 
one gets that the derivative of the restriction of $\mu$ to $\Ga_i$
is non-vanishing. One deduces that $\Ga_i$ is diffeomorphic to an annulus: 
it is foliated by circles and these circles 
admit a transverse orientation. 

For concluding the proof it remains fix $t\in(-\mu_1,\mu_1)$ regular value of $\mu$, then one fixes $\tilde X=X-tY$ and $\tilde Y=Y$.
One  chooses a compact neighborhood $\tilde U$ of $\gamma_i$ in $M$, which is a 
manifold with boundary, whose boundary is tranverse to $S$ and so that $\tilde U\cap S=\Ga_i$. By construction, 
$(\tilde U,\tilde X,\tilde Y)$ satisfies all the announced properties.
\end{proof}

\subsection{Prepared counter examples to Theorem~\ref{teoprincipal}}

\begin{defi}\label{d.prepared}
We say that $(U,X,Y,\Si,\cB)$ is a \emph{prepared counter example to Theorem~\ref{teoprincipal}} if 
\begin{enumerate}
 \item $(U,X,Y)$ is a counter example to Theorem~\ref{teoprincipal}
 \item There is $\mu_0>0$ so that $(U,X,Y)$ satisfies the conclusion of Lemma~\ref{l.annulus}:
 \begin{itemize}
\item for any $t\in [-\mu_0,\mu_0]$, the set of zeros of $ X-t Y$ in $ U$ consists 
precisely in $1$ periodic orbit $\gamma_t$ of $Y$;

\item for any $t\notin [-\mu_0,\mu_0]$, the set of zeros of $ X-t Y$ in $ U$ is empty; 
\item $\col( X, Y,U)$ is a $C^1$ annulus;
\item\label{i.mu} there is a $C^1$-diffeomorphism 
 $\varphi\colon\RR/\ZZ\times [-\mu_0,\mu_0]\to \col( X, Y, U)$ so that, for every 
 $t\in[-\mu_0,\mu_0]$, one has 
 $$\varphi(\RR/\ZZ\times\{t\})=\gamma_t.$$
\end{itemize}
\item $U$ is endowed with a foliation by discs; more precisly there is a smooth submersion 
$\Si\colon U\to \RR/\ZZ$ whose fibers $\Si_t=\Si^{-1}(t)$ are discs; furthermore, 
the vector field $Y$ is transverse to the fibers $\Si_t$. 
\item\label{i.period}  Each periodic orbit $\gamma_s$, $s\in[-\mu_0,\mu_0]$, of $Y$ cuts every  
disc $\Si_t$ in exactly one point. In particular
 the period  of $\gamma_s$ coincides with its return time on $\Si_0$ and is denoted $\tau(s)$, for 
 $s\in[-\mu_0,\mu_0]$.  
 
 Thus $s\mapsto \tau(s)$ is a $C^1$-map on $[-\mu_0,\mu_0]$. 
We require that the derivative of $\tau$ does not vanish on $[-\mu_0,\mu_0]$.
\item \label{i.basis} $\cB$ is a triple $(e_1,e_2,e_3)$ of $C^0$ vector fields on $U$ so that
\begin{itemize}
\item for any $x\in U$ $\cB(x)=(e_1(x),e_2(x),e_3(x))$ is a basis of $T_xU$. 
\item $e_3=Y$ everywhere
\item the vectors $e_1,e_2$ are tangent to the fibers $\Si_t$, $t\in \RR/\ZZ$.  In other words, 
$D\Si(e_1)=D\Si(e_2)=0$
\item The vector $e_1$ is tangent to $\col(X,Y)$ at each point of $\col(X,Y)$. 
\end{itemize}
\end{enumerate}
\end{defi}

\begin{lemm}
\label{l.prepared} 
If there exists a counter example $(U,X,Y)$ to Theorem~\ref{teoprincipal} then there is 
a prepared counter example $(\tilde U,\tilde X,\tilde Y,\Si, \cB)$ to Theorem~\ref{teoprincipal}.
\end{lemm}
\begin{proof} The two first items  of the definition of prepared counter example to 
Theorem~\ref{teoprincipal} are given by Lemma~\ref{l.annulus}.
For getting the third item, it is enough to shrink $U$.  
For getting item~(\ref{i.period}),  one replace $Y$ by $Y+bX$ for some $b\in\RR$, $|b|$ small enough. 
This does not change the orbits $\gamma_t$, as $X$ and $Y$ are both tangent to $\gamma_t$, 
but it changes its period.  Thus, this allows us to change the derivative of the period $\tau$ at $s=0$. 
Then one shrink again $U$ and $\mu_0$ so that the derivative  of 
$\tau$ will not vanish on $\col{(X,Y,U)}$. 

Consider any metric on $U$. 
For any point $x$ in the annulus $\col(X,Y,U)$ and contained  in the fiber $\Si_t$ we chose $e_1(x)$ as being a unit vector tangent to 
the segment $\col(X,Y)\cap \Si_t$. We extend $e_1$ as a continuous vector field on $U$ tangent to the fibers $\Si_t$.
We choose 
$e_2(x)$, for any $x$ in a fiber $\Si_t$  as being an  unit vector tangent to $\Si_t$ and orthogonal to $e_1(x)$. Now $e_3(x)=Y(x)$ is 
transverse to the plane spanned by $e_1(x),e_2(x)$ so that $\cB(x)=(e_1(x),e_2(x), e_2(x))$  is a basis of $T_xM$. This provides 
the basis announced in item~\ref{i.basis}. 
 \end{proof}

\begin{rema}\label{r.prepared} If $(U,X,Y,\Si, \cB)$ is a prepared counter example to 
Theorem~\ref{teoprincipal}, then for every $t\in(-\mu_0,\mu_0)$, $(U,X-tY,Y,\Si, \cB)$ is a prepared 
counter example to 
Theorem~\ref{teoprincipal}. 
\end{rema}

Whenever $(U,X,Y,\Si,\cB)$ is a prepared counter example to Theorem~\ref{teoprincipal},
we shall denote by $\cP$ the \emph{first return map},
defined on a neighborhood of $\col(X,Y)\cap\Sigma_0$ in $\Sigma_0$.

\begin{rema}\label{r.P}
As the ambient manifold is assumed to be orientable (see Remark~\ref{r.orientable}), the vector field $Y$ 
is normally oriented so that  the Poincar\'e map 
$\cP$ preserves the orientation. 
\end{rema}

\subsection{Counting the index of a prepared counter example}

\begin{defi}
Let $(U,X, Y,\Si, \cB)$  be a prepared counter example to Theorem~\ref{teoprincipal}. 
In particular, $U$ is a solid torus  ($C^1$-diffeomorphic to $\DD^2\times\RR/\ZZ$) and $\zero(X-tY)$, 
$t\in(-\mu_0,\mu_0)$, 
is an essential simple curve $\gamma_t$ isotopic to $\{0\}\times\RR/\ZZ$. 
An \emph{essential torus $T$ } is the image of a continuous map from the torus $\TT^2$ in the interior of $U$, disjoint 
from $\gamma_0=\zero(X)$ and homotopic, in $U\setminus \gamma_0$, to the boundary of 
a tubular neighborhood of 
$\gamma_0$. 
\end{defi}

In other words, $H_2(U\setminus\gamma_0,\ZZ)=\ZZ$, and   $T$ is essential if it is the generator of this second 
homology group.

We shall now describe how we use the basis $\cB$, which comes with a prepared counter example, 
and an essential 
torus $T$ to calculate the index. 

For each point $x\in U$, one can write $X(x)$ as a linear combination of the vectors $e_1(x),e_2(x)$ and $e_3(x)$. 
Notice that, since $e_3=Y$ everywhere, the $e_3$-coordinate $\mu$ of $X$ is a $C^1$ 
extension of the ratio function, $\mu$ 
that we introduced on $\col(X,Y)$ in item (\ref{d.ratio}) of Subsection~\ref{ss.commuting-first}. Therefore, there exists $C^1$ functions
$\alpha,\beta,\mu:U\to\RR$ such that

\begin{equation}
\label{e.xnabase}
X(x)=\alpha(x)e_1(x)+\beta(x)e_2(x)+\mu(x)e_3(x).
\end{equation}

For $x\notin \gamma_0$ one considers the vector
\begin{equation}\label{e.xnaesfera}
\cX(x)=\frac{1}{\sqrt{\alpha(x)^2+\beta(x)^2+\mu(x)^2}}\left(\alpha(x),\beta(x),\mu(x)\right)\in\SS^2.
\end{equation}
The map restriction $\cX|_T\colon T\to \SS^2$ has a topological degree,
which, by Lemma~\ref{l.Using the tubular neighboorhod to calculate the index}, coincides with 
$\operatorname{Ind}(X, U)$, 
for some choice of an orientation on $T$.

\subsection{The normally hyperbolic case}\label{partial}

In this section we illustrate our procedure by giving the very simple proof of 
Theorem~\ref{teoprincipal} in the case where $\col(X,Y,U)$ is furthermore assumed to be 
normally hyperbolic for the flow of $Y$. 

Here we shall prove 

\begin{lemm} 
\label{l.casoph}
Let $(U,X,Y,\Si,\cB)$ be a prepared counter example to Theorem~\ref{teoprincipal}.
Then, the first return map $\cP\colon\Sigma_0\to\Sigma_0$ of the flow of $Y$ satisfies:  
for every point $x$ of $\col(X,Y,U)\cap \Si_0$, the unique eigenvalue of the derivative of $\cP$ 
at $x$ 
is $1$.  
\end{lemm} 
\begin{proof}
The argument is by contradiction. Let us denote $x_t=\gamma_t\cap\Sigma_0$, $t\in [-\mu_0,\mu_0]$, 
(recall $\gamma_t=\zero(X-tY)$). We assume that the derivative of 
$\cP$ at some point of $x_{t_0}$ has some eigenvalue of  different from $1$. 

Notice that the first return map $\cP$ is the identity map in restriction to the segment 
$\col(X,Y)\cap\Sigma_0$. In particular, the derivative of $\cP$ at $x_t$ admits $1$ as an eigenvalue. 
Since $\cP$ preserves the orientation (see Remark~\ref{r.P}), the other eigenvalue is positive.

\begin{clai} There exists $\tilde U\subset U$, $\tilde X= X-tY$ 
and a prepared counter example to Theorem~\ref{teoprincipal} 
$(\tilde U,\tilde X,Y,\Si,\cB)$ for which the surface 
$\col(\tilde X,Y,\tilde U)$ is normally hyperbolic. 
\end{clai}
\begin{proof}
As the property of having a eingenvalue of mudulus different from $1$ is an open condition, there exists an interval
$[\mu_1,\mu_2]\subset[-\mu_0,\mu_0]$ on which the condition holds. Consider $t=\frac{\mu_1+\mu_2}2$ and $\tilde X=X-tY$. 

Then, one obtains a new prepared counter example to Theorem~\ref{teoprincipal} by replacing $X$ by 
$\tilde X$ (see Remark~\ref{r.prepared}); now, by shrinking 
$U$ one gets a tubular neighborhood $\tilde U$ of $\gamma_t$ so that 
$\col(\tilde X,Y,\tilde U)=\bigcup_{s\in[\mu_1,\mu_2]}\gamma_s$. 

Moreover, the derivative of $\cP$ at each point $x_s$, $s\in[\mu_1,\mu_2]$, has an 
eigenvalue different from $1$ in a
 direction tranverse to $\col(\tilde X,Y,\tilde U)\cap\Sigma_0$. 
 By compactness and continuity these eigenvalues 
 are uniformly far from $1$ so that $\col(\tilde X,Y,\tilde U)\cap\Sigma_0$ 
 is normally hyperbolic for 
 $\cP$.
 
Thus $\col(\tilde X,Y,\tilde U)$ is an invariant normally hyperbolic annulus for the flow of $Y$.
\end{proof}

By virtue of the above claim (up to change $X$ by $\tilde X$ and $U$ by $\tilde U$) one may assume that 
$\col(X,Y,U)$ is normally hyperbolic, and (up change $Y$ by $-Y$) one may assume that 
$\col(X,Y,U)$ is normally contracting. 

This implies that every periodic orbit $\gamma_t$ has a local stable manifold $W^s_Y(\gamma_t)$ which is a 
$C^1$-surface 
depending continuously on $t$ for the $C^1$-topology  and the collection of 
these surfaces build a $\cC^{0}$-foliation $\cF^s_Y$ tangent to a 
continuous plane 
field $E^s_Y$,  in a  neighborhood of $\col(X,Y,U)$. 
Furthermore, $E^s_Y$ is tangent to $Y$, and hence is tranverse to
the fibers of $\Si$. 

Up to shrink $U$, one may assume that $\cF^s_Y$ and $E^s_Y$ are defined on $U$. 

\begin{clai} There is a basis $\tilde \cB=(\tilde e_1,\tilde e_2,\tilde e_3)$ so that $(U,X,Y,\Si,\tilde\cB)$
is a prepared counter example to Theorem~\ref{teoprincipal} and $\tilde e_2$ is tangent to $E^s_Y$. 
\end{clai}
\begin{proof}
Choose $\tilde e_2$ as being a unit vector tangent to the intersection of $E^s_Y$ 
with the tangent plane of the fibers of 
$\Si$. It remains to choose $\tilde e_1$ 
tranverse to $e_2$ and tangent to the fibers of $\Si$ and tangent to $\col(X,Y)$ at every point of $\col(X,Y)$. 
\end{proof}

Up to change $\cB$ by the basis $\tilde\cB$ given by the claim above, we will now assume that $e_2$ is 
tangent to $E^s_Y$.  

\begin{clai} The vector field $X$ is tangent to $E^s_Y$.
\end{clai}
\begin{proof}
The flow of the vector field $X$ leaves invariant the periodic orbit $\gamma_t$ of $Y$ and $X$ commutes with $Y$. 
As a consequence, it preserves the stable manifold $W^s(\gamma_t)$ for every $t$.  
This implies that $X$ is tangent to the foliation 
$\cF^s_Y$ and therefore to $E^s_Y$. 
\end{proof}

Therefore, for every $x\in U\setminus \zero(X)$ the vector $\cX(x)\in\SS^2$ 
(see the notations in Equations~\ref{e.xnabase} and~\ref{e.xnaesfera}) belongs to the circle 
$\{x_1=0\}$. In particular, for any essential torus $T$ the map $\cX|_T\colon T\to \SS^2$ is not surjective, 
an thus has zero topological degree. This proves that $\operatorname{Ind}(X, U)$ vanishes, contradicting 
the fact that 
$(U,X,Y,\Si,\cB)$ is assumed to be a prepared counter example to Theorem~\ref{teoprincipal}.
\end{proof}

\section{Holonomies, return time, and the normal component}\label{Coordinates}
In the whole section, $(U,X,Y,\Si,\cB)$ is a prepared counter example to Theorem~\ref{teoprincipal}.

\paragraph{Definitions}
Recall that $\Sigma_{t}=\Si^{-1}(t)$, 
$t\in \RR/\ZZ$, is a
family of cross section, each $\Sigma_t$ is diffeomorphic to a disc, 
and we identify $\Sigma_0$ with the unit disc $\DD^2$.

\begin{defi}\label{d.holonomy}
Consider $t\in \RR$.  Consider $x\in \Sigma_0$ and $y\in \Sigma_{t}$. We say that 
$y$ is \emph{the image by holonomy of 
$Y$ over the segment $[0,t]$}, and we denote $y=P_{t}(x)$, if there exists a continuous path 
$x_r\in U$, $r\in[0,t]$, so that 
$\Si(x_r)=r$, $x_0=x$, $x_t=y$, and for every $r\in[0,t]$ the point $x_r$ belongs to the $Y$-orbit of $x$.
\end{defi}

The holonomy map $P_{t}$ is well defined in a neighborhood of $\col(X,Y,U)\cap\Sigma_0$ and is a 
$C^1$ local diffeomorphism. 

If $t=1$ then $P_1$ is the first return map $\cP$ (defined before Remark~\ref{r.P}) of the flow of $Y$ on the cross section 
$\Sigma_0$. 

\begin{rema}
With the notation of Definition~\ref{d.holonomy}, there is a unique continuous function 
$\tau_x\colon[0,t]\to\RR$ so that $\tau_x(0)=0$ and $x_r=Y_{\tau_x(r)}(x)$ for every $r\in[0,t]$.  

We denote $\tau_{t}(x)= \tau_x(t)$ and we call it the \emph{transition time from $\Sigma_0$ to $\Sigma_{t}$}. 
The map $\tau_{t}\colon\Sigma_0\to\RR$ is a $C^1$ map and by definition one has
\begin{equation}
\label{e.holonmy-times}
 P_{t}(x)=Y_{\tau_{t}(x)}(x)
\end{equation}

We denote $\tau=\tau_{1}$ and we call it the first return time of $Y$ on $\Sigma_0$. 
\begin{rema}  In Definition~\ref{d.prepared} item~\ref{i.period} 
we defined $\tau(s)$ as the period of $\gamma_s$;  in the notation above, it coincides with $\tau(x_s)$ where
$x_s=\gamma_s\cap \Si_0$. 
\end{rema}

In this case, Equation~\ref{e.holonmy-times} takes the special form 
\begin{equation}\label{e.return-times}
 \cP(x)=Y_{\tau(x)}(x)
\end{equation}
\end{rema}

\subsection{The normal conponent of $X$}

\begin{defi}
For every $t$ and every $x\in \Sigma_t$ we define \emph{the normal component of $X$}, which we denote by $N(x)$,
the  projection of $X(x)$ on $T_x\Sigma_t$ parallel to $Y(x)$.
\end{defi} 

Thus $x\mapsto N(x)$ is a $C^1$-vector field tangent to the fibers of $\Si$
and which vanishes precisely on $\col(X,Y,U)$. 

Moreover, in the basis $\cB$, $N(x)=\alpha(x)e_1(x)+\beta(x)e_2(x)$ (see Equation~\ref{e.xnabase}), 
and we have the 
following formula
$$X(x)=N(x)+\mu(x)Y(x),$$
for every $x\in U$.

\paragraph{The first return map and the derivative of the first return time.}

The goal of this paragraph is the proof of Corollary~\ref{derivative.return.time} which claims the following formula
\begin{equation}
\label{e.formula}
-D\tau(x)N(x)=\mu((\cP(x)))-\mu(x)
\end{equation}
relating the derivative of the return time function $\tau:\Sigma_0\to(0,+\infty)$, the normal component $N$ and the first return map $\cP$, at every point $x\in\Sigma_0$. 
This formula will be crucial for transfering informations on the normal vector field $N$ to the first return map $\cP$.
At the end of this section, we shall use the angular variation of $N$ (see the precise formulation in Corollary~\ref{c.gira})
together with formula (\ref{e.formula}) to
prove that the derivative of $\cP$ at any fixed point is the identity. 
 
The geometrical idea for proving (\ref{e.formula}) 
is the following: \begin{itemize}
\item at one hand, if $x\in\Sigma_0$, one has that the difference between the vectors $D\cP(x)N(x)$ and $DY_{\tau(x)}(x)N(x)$ is parallel to $Y(\cP(x))$,
and the proportion is given exactly by $D\tau(x)N(x)$ (this is a classical fact, which we state precisely below). 
\item On the other hand, $X$ is equal to $N+\mu Y$  and is invariant under $DY_t$, for any $t$.\end{itemize}
One gets (\ref{e.formula}) by combining these two facts.

\begin{figure}[h]
\centering
\includegraphics[width=270pt,height=180pt]{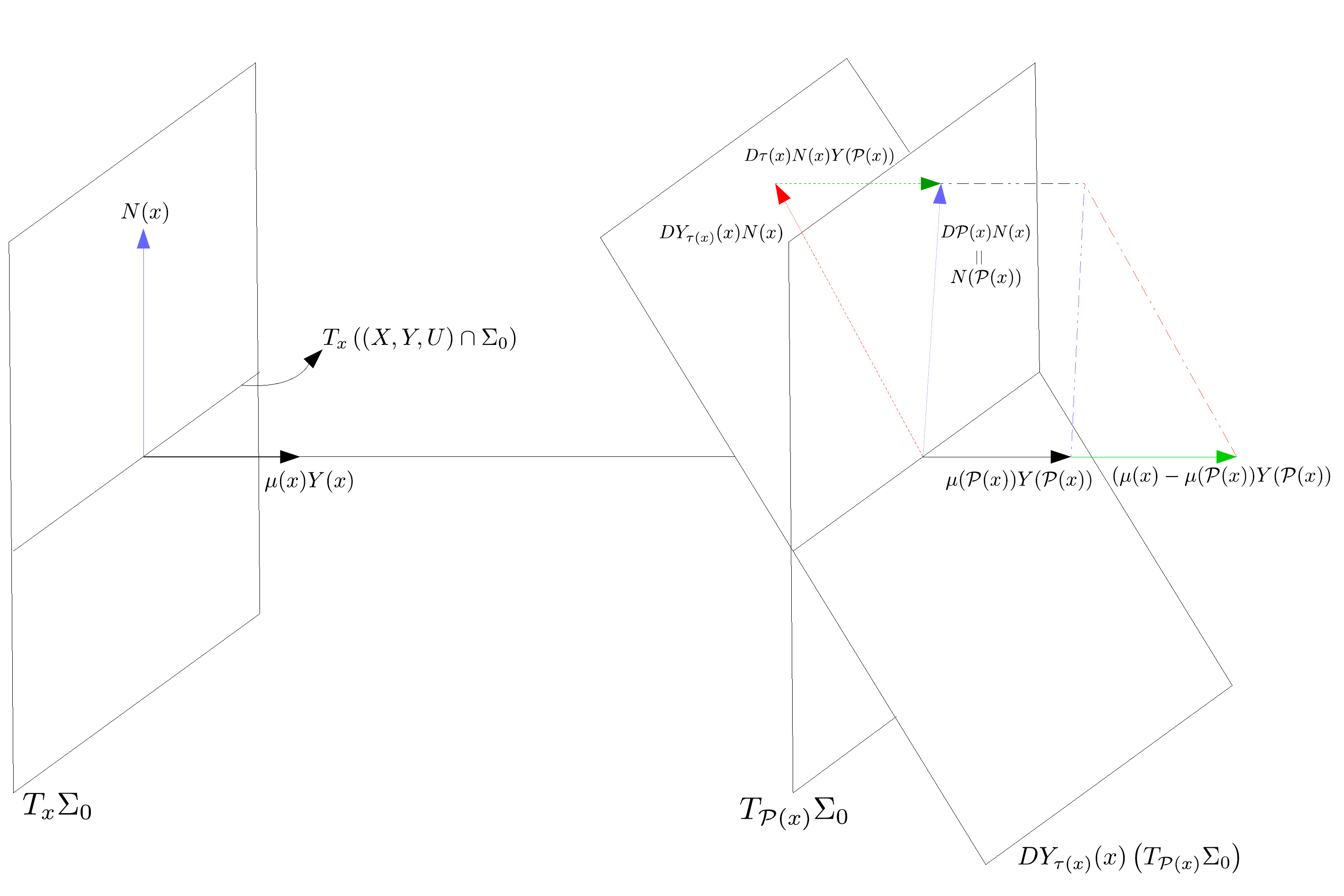}
\caption{
Geometric proof of formula (\ref{e.formula}):
\footnotesize
$$N(\cP(x))+\mu(\cP(x))Y(\cP(x))= X(\cP(x))=DY_{\tau(x)}(x)N(x)+\mu(x)Y(\cP(x))$$
$$\mbox{thus }\quad D\tau(x)N(x)Y(\cP(x))=(\mu(x)-\mu(\cP(x)))Y(\cP(x))$$
}
\label{f.formula}
\end{figure}

Let us proceed with the formal proof. The first step is the (classical) result below. The proof is an elementary and simple application of the flow box theorem, so we omit. 

\begin{lemm}
\label{l.diferenca.das.derivadas}
For every $t\in\RR$   $x\in\Sigma_0$ and every $v\in T_x\Sigma_0$ one has, 
\begin{equation}
DP_{t}(x)v-DY_{\tau_{t}(x)}(x)v=D\tau_{t}(x)v.Y(P_{t}(x)).
\end{equation}
In particular, for every $x\in\Sigma_0$ and $v\in T_x\Sigma_0$, one has
\begin{equation}
D\cP(x)v-DY_{\tau(x)}(x)v=D\tau(x)v.Y(\cP(x)).
\end{equation}
\end{lemm}

Now, recall that 
$$X(x)=N(x)+\mu(x)Y(x)$$ and since 
$$DY_{\tau(x)}X(x)=X(Y_{\tau(x)})=X(\cP(x)),$$ 
we obtain
\begin{equation}
\label{e.equmcor}
X(\cP(x))=DY_{\tau(x)}(x)N(x)+\mu(x)Y(\cP(x)).
\end{equation}
On the other hand, by Lemma~\ref{l.diferenca.das.derivadas}
$$DY_{\tau(x)}(x)N(x)=D\cP(x)N(x)-D\tau(x)N(x)Y(\cP(x)).$$
As 
$$X(\cP(x))=N(\cP(x))+\mu(\cP(x))Y(\cP(x)),$$ 
we conclude that
\begin{eqnarray*}
D\cP(x)N(x)-D\tau(x)N(x))Y(\cP(x))+\mu(x)Y(\cP(x))\\
=N(\cP(x))+\mu(\cP(x))Y(\cP(x)),
\end{eqnarray*}
and so 
$$(\mu(x)-\mu(\cP(x))-D\tau(x)N(x)))Y(\cP(x))+D\cP(x)N(x)-N(\cP(x))=0.$$
Since $Y$ is transverse to the tangent space of $\Sigma_0$, the vectors $Y(\cP(x))$ and $D\cP(x)N(x)-N(\cP(x))$ are linearly independent. As a consequence, one obtains that 
$D\cP(x)N(x)=N(\cP(x))$ and

\begin{coro} 
\label{derivative.return.time}
Let $x\in\Sigma_0$. Then, $$-D\tau(x)N(x)=\mu((\cP(x)))-\mu(x).$$
\end{coro}

Moreover, applying the same argument with the time $t$ holonomy $P_t$ in place of the first return map $\cP$, one obtain the invariance of $N$ under the holonomies.

\begin{lemm}
\label{l. u invariante}
For $t\in\RR$ and every $x\in\Sigma_0$ one has   
$$
DP_{t}(x)N(x)=N(P_{t}(x)).
$$
\end{lemm}

\paragraph{The normal component of $X$ and the index of $X$}

Let  $(U,X,Y,\Si,\cB)$ be a prepared counter example to Theorem~\ref{teoprincipal} and $N$ be
the normal component of $X$. 

Let us define, for $x\in U\setminus \col{(X,Y,U)}$, 
$$\cN(x)=\frac{1}{\sqrt{\alpha(x)^2+\beta(x)^2}}\left(\alpha(x),\beta(x)\right)\in\SS^1\subset \SS^2,$$
where $\SS^1$ is the unit circle of the plane $\RR^2\times\{0\}\subset \RR^3$. 

Recall that $U$ is homeomorphic to the solid torus so that its first homology group $H_1(U,\ZZ)$ is
isomorphic to $\ZZ$, by an isomorphism
sending the class of $\gamma_0$, oriented by $Y$, on $1$.
Let $U^+$ and $U^-$ be the two connected components of $U\setminus\col{(X,Y,U)}$. These are also solid tori, 
and the inclusion in $U$ induces isomorphisms of the first homology groups which
allows us to identify $H_1(U^\pm,\ZZ)$ with $\ZZ$.

\begin{defi}  
We call \emph{linking number} of $X$ with respect to $Y$ in $U^+$ (resp. in $U^-$) and we denote it by 
$\ell^+(X,Y)$ (resp. $\ell^-(X,Y)$) the integer defined as follows: 
the continuous map $\cN\colon U^\pm\to \SS^1$ induces morphisms on the homology groups 
$H_1(U^\pm,\ZZ)\to H_1(\SS^1,\ZZ)$. As these groups are all identified with $\ZZ$, these morphisms consist in 
the multiplication by an integer $\ell^\pm(X,Y)$. 
\end{defi}

In other words, consider a closed curve $\sigma\subset U^+$ homotopic in $U$ to $\gamma_0$.  Then $\ell^+(X,Y)$ is 
the topological degree of the restriction of $\cN$ to $\sigma$. 

\begin{prop}\label{p.link}Let  $(U,X,Y,\Si,\cB)$ be a prepared counter example to 
Theorem~\ref{teoprincipal}.
Then
 $$|\operatorname{Ind}(X, U)|=|\ell^+(X,Y)-\ell^-(X,Y)|$$
\end{prop}
\begin{proof}
 Consider a tubular neighborhood of $\gamma_0$ whose boundary  is an essential torus $T$ which cuts $\col{(X,Y,U)}$ 
 transversely and along exactly two
 curves $\sigma_+$ and $\sigma_-$.  Then the map $\cX$ on  $T$ takes the value $N\in\SS^2$ 
 (resp. $S\in\SS^2$) exactly  on $\sigma_+$ (resp. $\sigma_-$), where $N$ and 
 $S$ are the points on $\SS^2$ corresponding to 
 $e_3=Y$ and $-e_3$. 
 
 We identify $T$ with $\TT^2=\RR/\ZZ\times \RR/\ZZ$ so that $\sigma_-$ and $\sigma_+$ correspond to 
 $\{\frac 12\}\times\RR/\ZZ$ and $\{0\}\times \RR/\ZZ$ respectively. It remains to apply Corollary~\ref{c.torus}
 to $\Phi=\cX$ and $\varphi=\cN$, noticing that $\cN$ is the projection of $\cX$ on $\SS^1$ 
 along the meridians. 
 This gives the announced formula. 
\end{proof}

As a direct consequence of Proposition~\ref{p.link} one gets
\begin{coro}
\label{c.link} 
Let  $(U,X,Y,\Si,\cB)$ be a prepared counter example to Theorem~\ref{teoprincipal}, then
$$(\ell^+(X,Y),\ell^-(X,Y))\neq (0,0)$$ 
\end{coro}

\subsection{Angular variation of the normal component}

\paragraph{Definition of the angular variation}
Recall that, for any $x\in U\setminus \col(X,Y)$, we defined $\cN(x)\in \SS^1$ as the expression in the basis $\cB(x)$ 
of the renormalization of the normal component $N(x)$. 

Consider a path $\eta\colon [0,1]\to U\setminus \col(X,Y)$.  We define the \emph{angular variation of $\cN$ along $\eta$} is as being
$$\tilde \cN(\eta(1))-\tilde\cN(\eta(0))$$ 
where $\tilde\cN(\eta(t))$ is a lift of the path $t\mapsto\cN(\eta(t)\in \SS^1$ on the universal cover 
$\RR\mapsto \SS^1\simeq \RR/2\pi\ZZ$. 
 
By a practical abuse of language, if there is no ambiguity on  the basis $\cB$, we will call \emph{angular variation of $N$ along $\eta$} the angular variation of $\cN$ along $\eta$. 

\paragraph{Angular variation of the normal component $N$ along the $Y$-orbits}\label{ss.angular}
We denote $\{x_t\}=\gamma_t\cap\Sigma_0$. For every pair of points $x,y\in\Sigma_0$, 
we denote the segment of straight line joinning x and y and contained in $\Sigma_0$ by
$[x,y]$  (for some choice of coordinates on $\Sigma_0$).

\begin{lemm}
\label{l.naogira} For any $K>0$ there is a neighborhood $V_K$ of 
$\gamma_0$ with the following  property. 

Consider $x\in V_K\cap \Sigma_0$ and $u\in T_x \Sigma_0$ a unit vector, and write $u=u_1 e_1(x)+u_2 e_2(x)$. 
Consider $t\in[0,K]$ and $v=DP_{t}(u)\in T_{P_{t}(x)}\Sigma_t$ the image of $u$ 
by the derivative of the holonomy. Write $v= v_1 e_1(P_{t}(x))+v_2 e_2(P_{t}(x))$. 

Then 
$$
\left(\frac{u_1}{\sqrt{u_1^2+u_2^2}},\frac{u_2}{\sqrt{u_1^2+u_2^2}}\right)
\neq -\left(\frac{v_1}{\sqrt{v_1^2+v_2^2}},\frac{v_2}{\sqrt{v_1^2+v_2^2}}\right) 
$$
\end{lemm}
\begin{proof}
Assuming, by contradiction, that the conclusion does not hold
we get $y_n\in\Sigma_0$, unit vectors $u_n\in T_{x_n}\Sigma_0$ and $t_n\in[0,K]$ so that $y_n$ 
tends to $x_0=\gamma_0\cap \Sigma_0$, 
$u_n$ tends to a 
unit vector $u$ in $T_{x_0}\Sigma_0$,
$t_n$ tends
to $t\in[0,K]$ and the image $v_n$ of the vector $u_n$, expressed in the basis 
$\cB$, is collinear to $u_n$ with the opposite direction.

Then $DP_{t}(x_0)u$ is a vector that, 
expressed  in the basis $\cB$, is collinear to $u$ with the opposit direction. 
In other words, $u$ is an eigenvector of $DP_{t}(x_0)$, with a negative eigenvalue. 

However, for every $t\in\RR$  the vector $e_1$ is an eigenvector of $DP_{t}(x_0)$, 
with a positive eigenvalue, and $DP_{t}(x_0)$ preserves the orientation, leading to a contradiction. 
\end{proof}


\begin{lemm}
\label{l.segment} 
For every $N\in\N$ there exists a neighborhood $O_N\subset \Sigma_0$ of $x_0$
so that if $x\in O_N\setminus\col(X,Y,U)$
then the segment of straight line $[x,\cP^N(x)]$ is disjoint from $\col(X,Y,U)$. 
\end{lemm}
\begin{proof}Assume that there is a sequence of points $y_n\to x_0$, $y_n\notin \col(X,Y,U)$ 
so that the segment 
$[y_n,\cP^N(y_n)]$ intersect $\col(X,Y,U)\cap \Sigma_0$ at some point $z_n$. Recall that 
$\col(X,Y,U)\cap \Sigma_0$ consists in fixed points of the first return map $\cP$.  
In particular, $z_n$ is a fixed point of $\cP^N$. 

The image of the segment  $[z_n,y_n]$  is a $C^1$ curve joining $z_n$ to $\cP^N(y_n)$. 
Notice that the segments $[z_n,y_n]$ and $[z_n,\cP^N(y_n)]$ are contained 
in the segment $[y_n,\cP^N(y_n)]$ and oriented in opposite direction. 
One deduces that there is a point $w_n$ in $[y_n,z_n]$ so that the image under the derivative 
$D\cP^N$ of the unit vector $u_n$ directing this segment is on the form $\lambda_n u_n$ with $\lambda_n<0$. 

Since $y_n$ tends to $x_0$, one deduces that $D\cP^N(x_0)$ has a negative 
eigenvalue. This contradicts the fact that both eigenvalues of 
$D\cP^N(x_0)$ are positive and completes the proof. 
\end{proof}

Recall that $U^+$ and $U^-$ are the connected components of $U\setminus \col{(X,Y,U)}$. 

\begin{coro} If $x\in U^\pm\cap \Sigma_0\cap O_3$  let   $\theta_x\colon \RR/\ZZ\to U$ 
be the curve  obtained  
by concatenation of the $Y$-orbit segment from $x$ to $\cP^2(x)$ and the straight 
line segment $[\cP^2(x),x]$, that is:  
\begin{itemize}
 \item for $t\in[0,\frac 12]$, $\theta_x(t)= Y_{2t.\tau_{2}(x)}(x)$ where $\tau_{2}$ is 
 the transition time from $x$ to $\cP^2(x)$
 \item for $t\in[\frac 12,1]$, $\theta_x(t)=(2-2t)\cP^2(x) +(2t-1)x$.
\end{itemize}
Then $\theta_x$ is a closed curved contained in $U^\pm$ and whose homology class in $H_1(U^\pm,\ZZ)=\ZZ$ is 
$2$. 
\end{coro}
\begin{proof}The unique difficulty here is that the curve don't cross the colinearity locus
$\col{(X,Y,U)}$ and that is given by Lemma~\ref{l.segment}.
\end{proof}

The next corollary is one of the fundamental arguments of this paper. 

\begin{coro}\label{c.gira} 
Let  $(U,X,Y,\Si,\cB)$ be a prepared counter example to Theorem~\ref{teoprincipal}, and assume that 
$\ell^+(X,Y)\neq 0$. 

Consider $x\in U^+\cap \Sigma_0\cap O_3$.  Then the angular variation of the vector $\cN(y)$ for 
$y\in [x,\cP^2(x)]$ is strictly larger than $2\pi$ in absolute value. In particular, 
$$\cN([x,\cP^2(x)])=\SS^1.$$

The same statement holds in $U^-$ if $\ell^-(X,Y)\neq 0$. 
\end{coro}
\begin{proof} Lemma~\ref{l.naogira} implies that the angular variation of $\cN(\theta_x(t))$ is contained
in $(-\pi,\pi)$  for $t\in[0,\frac 12]$. 
However, the topological degree of the map $\cN\colon \theta_x\to \SS^1$ is $2\ell^+(X,Y)$
which has absolute value 
at least $2$. Thus the angular variation of $\cN$ on the segment $[x,\cP^2(x)]$ 
is (in absolute value) at least $3\pi$ concluding. 
 \end{proof}

As an immediate consequence one gets
\begin{coro}\label{c.nofixed} If $\ell^+(X,Y)\neq 0$, then $\cP^2$ has no fixed points in 
$O_3\cap U^+$.
\end{coro}

\paragraph{The return map at points where $N$ is pointing in opposite directions.}

We have seen in the proof of  Corollary \ref{c.gira} that the vector $\cN$ has an angular variation
larger than $3\pi$ along the segment $[x,\cP^2(x)]$, as $x\in\Si_0$
approaches $x_0$. In this section we use the large angular variation of $\cN$  
for establishing  a relation between the return map $\cP$, the return time $\tau$, and the coordinate 
$\mu$ of $X$ in the $Y$ direction. 

Recall that, for every $x\in U\setminus \col(X,Y,U)$,  $\cN(x)$ is a unit vector contained in $\SS^1$, 
unit circle of $\RR^2$. 

\begin{lemm}
 \label{l.lemamatador}
Assume that there exists sequences 
$$q_n,\overline{q}_n\in\Si_0\setminus \col(X,Y,U)$$  
converging to $ x_0$, such that the following two 
conditions are satisfied:
\begin{enumerate}
 \item $\cN(q_n)\to (1,0)$ and 
 $\cN(\overline{q}_n)\to (-1,0)$, as $n\to+\infty$,
 \item $(\mu(\cP(q_n))-\mu(q_n))(\mu(\cP(\overline{q}_n))-\mu(\overline{q}_n))\geq 0$, for every $n$.
\end{enumerate}
Then, $D\tau(x_0)e_1(x_0)=0$.
\end{lemm}
\begin{proof}
Recall that $N(x)=\alpha(x)e_1(x)+\beta(x)e_2(x)$ and 
$$\cN(x)=\frac1{\sqrt{\alpha(x)^2+\beta(x)^2}}(\alpha(x),\beta(x))$$
for $x\in U\setminus\col(X,Y,U)$. 

By Corollary \ref{derivative.return.time}, we have 
\begin{equation}
\label{absurdo.chegando}
-D\tau(q_n)\frac{N(q_n)}{\sqrt{\alpha(q_n)^2+\beta(q_n)^2}}=\frac{\mu(\cP(q_n))-\mu(q_n)}{\sqrt{\alpha(q_n)^2+\beta(q_n)^2}},
\end{equation} 

and

\begin{equation}
\label{absurdo.quase.chegando}
-D\tau(\overline{q}_n)\frac{N(\overline{q}_n)}{\sqrt{\alpha(\overline{q}_n)^2+\beta(\overline{q}_n)^2}}=
\frac{\mu(\cP(\overline{q}_n))-\mu(\overline{q}_n)}{\sqrt{\alpha(\overline{q}_n)^2+\beta(\overline{q}_n)^2}}.
\end{equation} 

Multiplying side by side Equations \ref{absurdo.chegando} and \ref{absurdo.quase.chegando} 
and using the second assumption of the lemma, we get 

$$D\tau(q_n)\frac{N(q_n)}{\sqrt{\alpha(q_n)^2+\beta(q_n)^2}}D\tau(\overline{q}_n)\frac{N(\overline{q}_n)}
{\sqrt{\alpha(\overline{q}_n)^2+\beta(\overline{q}_n)^2}}\geq 0.$$ 

Notice that the first assumption of the lemma is equivalent to
$$\frac{N(q_n)}{\sqrt{\alpha(q_n)^2+\beta(q_n)^2}}\to e_1(x_0)$$
and 
$$\frac{N(\overline{q}_n)}{\sqrt{\alpha(\overline{q}_n)^2+\beta(\overline{q}_n)^2}}\to- e_1(x_0).$$
Since $q_n,\overline{q}_n\to x_0$, from the continuity of $D\tau$, we conclude 
that 
$$0\geq -\left(D\tau(x_0)(e_1(x_0))\right)^2\geq 0,$$ which completes the proof.
\end{proof}

Since we assumed $(D\tau(x_0) (e_1(x_0)) \neq 0$
(item~(\ref{i.period}) of the definition of a prepared counter example to 
Theorem~\ref{teoprincipal}), one gets the following corollary: 

\begin{coro} \label{c.lemamatador}
Let $(U,X,Y,\Si,\cB)$ is a prepared counter example to Theorem~\ref{teoprincipal}.
Assume that there exists sequences $q_n,\overline{q}_n\in \Si_0\setminus\col(X,Y,U)$ converging to  $x_0$, 
such that 
$\cN(q_n)\to (1,0)$ and 
 $\cN(\bar q_n)\to(-1,0)$, as $n$ tends to $+\infty$.
Then
 
 $$(\mu(\cP(q_n))-\mu(q_n))(\mu(\cP(\overline{q}_n))-\mu(\overline{q}_n)) <0, $$
 for every $n$ large enough.
\end{coro}

\subsection{The case $D\cP(x_0)\neq Id$}
In this section, $(U,X,Y,\Si,\cB)$ is a prepared counter example to Theorem~\ref{teoprincipal} so that 
the derivative of the first return map $\cP$ at the point $x_0=\Sigma_0\cap \gamma_0$ is not the identity map.
Recall that $D\cP(x_0)$ admits an eigenvalue equal to $1$ directed by $e_1(x_0)$, 
has no eigenvalues  different from
$1$ and is orientation preserving.  

Recall that $\col{(X,Y,U)}$ cuts the solid torus $U$ in two components $U^+$ and $U^-$.  
Let denote $\Sigma_+=\Sigma_0\cap U^+$ and $\Sigma_-=\Sigma_0\cap U^-$. 

\begin{lemm}\label{l.sign} Let $(U,X,Y,\Si,\cB)$ be a prepared counter example to 
Theorem~\ref{teoprincipal} so that 
the derivative of the first return map $\cP$ at the point $x_0$ is not the identity map.

There is a neighborhood $W$ of $x_0$ in $\Sigma$ so that the map $\mu(\cP(x))-\mu(x)$ restricted 
to $W$ vanishes only on $\col{(X,Y,U)}$. 

More precisely, \begin{itemize}
                 \item $(\mu(\cP(x))-\mu(x))(\mu(\cP(y))-\mu(y))>0$ if $x \mbox{ and } y\in W\cap \Si_+$  and if  $x \mbox{ and } y\in W\cap \Si_-$
                 \item $(\mu(\cP(x))-\mu(x))(\mu(\cP(y))-\mu(y))<0$ if $x\in W\cap \Si_+ \mbox{ and } y\in W\cap \Si_-$ and 
                 if $x\in W\cap \Si_- \mbox{ and } y\in W\cap \Si_+$
                 \item $(\mu(\cP(x))-\mu(x))=0$ and  $x\in W$ if and only if $x\in\col{(X,Y,U)}$.
                \end{itemize}

\end{lemm}
\begin{proof}The derivative $D\mu(x_0)(e_1(x_0))$ do not vanish (item~\ref{i.mu} of 
Definition~\ref{d.prepared}).  Thus
the kernel of $D\mu(x_0)$ is tranverse to $e_1(x_0)$.  The derivative $D\cP(x_0)$ admits $e_1(x_0)$ 
as an eigenvector (for the eigenvalue $1$)
and has no eigenvalue different from $1$ and is not the identity.  
This implies that the kernel of $D\mu(x_0)$ 
is not an eigendirection of $D\cP(x_0)$. Notice that the derivative at $x_0$ of the map $\mu(\cP(x))-\mu(x)$ is
$D\mu(x_0) D\cP(x_0)-D\mu(x_0)$.  As we have seen above the kernel of $D\mu(x_0) D\cP(x_0)$ is different from the kernel
of $D\mu(x_0)$ so that $D\mu(x_0) D\cP(x_0)\neq D\mu(x_0)$.

As a consequence the derivative of the function 
$x\mapsto\mu(\cP(x))-\mu(x))$ does not vanish at $x=x_0$. 

Thus $\{x\in \Si, \mu(\cP(x))-\mu(x))=0\}$ is a codimension $1$ submanifold in a neighborhood of 
$x_0$ in $\Sigma_0$,
and this submanifold contains $\col{(X,Y,U)}$, because $\col{(X,Y,U)}\subset Fix(\cP)$.  
Therefore these submanifolds coincide in the 
neighborhood of $x_0$, concluding.
\end{proof}

We are now ready to prove the following proposition: 

\begin{prop}\label{p.derivada}If $(U,X,Y,\Si,\cB)$ is a prepared counter example to Theorem~\ref{teoprincipal} then 
$D\cP(x)$ is the identity map for every $x\in\col{(X,Y,U)}\cap\Si_0$. 
\end{prop}
\begin{proof}Up to exchange $+$ by $-$, we assume that $\ell^+(X,Y)\neq 0$. 
Then Corollary~\ref{c.gira} implies that there are sequences $q_n,\bar q_n\in \Si_+$ tending to $x_0$ and so that 
$\cN(q_n)=(1,0)$ and $\cN(-q_n)=(-1,0)$. 
More precisely Corollary~\ref{c.gira} implies that for any $x\in\Si_+$ close enough to $x_0$ the segment $[x,\cP^2(x)]$
contains points $q,\bar q$ with $\cN(q)=(1,0)$ and $\cN(\bar q)=(-1,0)$. Now Lemma~\ref{l.segment} implies that the 
segment is contained in $\Si_+$, concluding.

Now Corollary~\ref{c.lemamatador} implies that, for $n$ large enough, 
the sign of the map $(\mu(\cP(x))-\mu(x))$
is different on $q_n$ and on $\bar q_n$. One concludes with Lemma~\ref{l.sign} 
which says that this sign cannot change
if the derivative $D\cP(x_0)$ is not the identity. Thus we proved $D\cP(x_0)=Id$.  

Now if $x\in\col{(X,Y,U)}$ then $x$ is of the form $x=x_t=\gamma_t\cap \Sigma_0$.  
According to Remark~\ref{r.prepared}
$(U,X-tY,Y,\Si,\cB)$ is also a prepared counter example to Theorem~\ref{teoprincipal}, and the 
linking number of
$\ell^+(X-tY,Y)$ is the same as the linking number of $\ell^+(X,Y)$, and therefore is not vanishing. 
Notice that the first return map $\cP$ is not affected by the substitution of the pair $(X,Y)$ by the pair $(X-tY,Y)$ but the point 
$x_0$ is now replaced by the point $x_t$.
Therefore the argument above establishes 
that $D\cP(x_t)=Id$ for every $x_t\in \col{(X,Y,U)}\cap\Si_0$.
\end{proof}

\section{Proof of Theorem \ref{teoprincipal}}\label{s.proof}
In the whole section, $(U,X,Y,\Si,\cB)$ is a prepared counter example to Theorem~\ref{teoprincipal}. 
According to Corollary~\ref{c.link} one of the linking numbers $\ell^+(X,Y)$, $\ell^-(X,Y) $
does not vanish, so that, up to exchange $+$ with $-$, one may assume $\ell^+(X,Y)\neq 0$.
According to Proposition~\ref{p.derivada} 
the derivative $D\cP(x)$ is the identity map for every $x\in\col{(X,Y,U)}\cap\Si_0$.

This means that, in a neighborhood of $\col{(X,Y,U)}$, the diffeomorphism $\cP$ is $C^1$ 
close to the identity map. The techniques introduced in \cite{Bonatti_esfera} and \cite{Bonatti_geral}
allow to compare the diffeomorphism $\cP$ with the vector field $\cP(x)-x$, and we will analyse the behavior 
of this vector field. 
We will first show that, in the neighborhood of $\col{(X,Y,U)}$ these vectors  are almost tangent to the kernel of 
$D\mu$. As the fibers of $D\mu$ are tranverse to $\col{(X,Y,U)}$ we get a topological dynamics of 
$\cP$ similar to the partially hyperbolic case of Section \ref{partial}.  We will end contradicting 
Corollary \ref{c.gira}.


\paragraph{Quasi invariance of the map $\mu$ by the first return map.}
Recall that $\mu$ is the coordinate of $X$ in the $Y$ direction: $X(x)=N(x)+\mu(x)Y(x)$.
The aim of this section is to prove

\begin{lemm}\label{l.angle} If $x_n\in \Sigma_+$ is a sequence of points tending to $x\in\col{(X,Y,U)}$ and if 
$v_n\in T_{x_n}\Sigma_+$ is the unit tangent vector directing the segment $[x_n,\cP(x_n)]$ then $D\mu(x_n)(v_n)$ 
tends to $0$.
\end{lemm}
According to Remark~\ref{r.prepared}, it is enough to prove Lemma~\ref{l.angle} in the case 
$x=x_0=\gamma_0\cap\Sigma_0$. Lemma~\ref{l.angle} is now a straighforward consequence 
of the following lemma

\begin{lemm}\label{l.angle2}
$$\lim_{x\to x_0, x\in\Sigma_+}\frac{\mu(\cP(x))-\mu(x)}{d(\cP(x),x)}=0,$$
 where $d(\cP(x),x)$ denotes the distance between $x$ and $\cP(x)$. 
\end{lemm}
Note that Lemma~\ref{l.sign} implies that this estimative could not hold if the derivative $D\cP(x_0)$ was not the identity. 
Here we use Proposition~\ref{p.derivada} which asserts that the derivative of $\cP$ at the fixed points is the identity.

\begin{proof}
Fix $\eps>0$ and let us prove that $\frac{|\mu(\cP(x))-\mu(x)|}{d(\cP(x),x)}$ is smaller than $\eps$ 
for every $x$ close to $x_0$ in $\Sigma_+$. 
Recall that, according to Lemma~\ref{l.segment}, there is a neighborhood $O_2$ of $\gamma_0$ so that 
if $x\in O_2^+= O_2\cap \Sigma_+ $ then the segment of straight line $[x,\cP^2(x)]$ is contained in 
$\Sigma_+$. Furthermore, Corollary \ref{c.gira} says that $\cN|_{[x,\cP^2(x)]}$ is surjective onto $\SS^1$ (unit circle in $\RR^2$). 
In particular, there are points 
$q_x,\bar q_x\in [x,\cP^2(x)]$ so that $\cN(q_x)=(1,0)$ and $\cN(\bar q_x)=(-1,0)$.

According to Corollary~\ref{c.lemamatador} one gets
\begin{equation}
\label{e.ftemdoissinais}
(\mu(\cP(q_x))-\mu(q_x))(\mu(\cP(\overline{q}_x))-\mu(\overline{q}_x)) <0,
\end{equation}
for every $x\in O_2^+$. 

The diffeomorphism $\cP$ is $C^1$-close to the identity in a small neighborhood of $x_0$
Now  \cite{Bonatti_esfera} (see also  \cite{Bonatti_geral}) implies that there is a neigborhood 
$V_1$ of $x_0$ in $\Sigma_0$ so that if $x\in V_1$ then 
$$\|(\cP(x)-x)-(\cP(y)-y)\|< \frac{1}{2} \|\cP(x)-x\|,$$ for every 
$y$ with $d(x,y)<3\|\cP(x)-x\|$. In particular, 
$$\|\cP^2(x)-\cP(x)\|<\frac{3}{2}\|\cP(x)-x\|<2\|\cP(x)-x\|,$$ 
and thus $\|\cP^2(x)-x\|<3\|\cP(x)-x\|$.

Consider the function $f(x)=\mu(\cP(x))-\mu(x)$. Since $D\cP(x_0)=Id$, we have that $Df(x_0)=0$. 
As a consequence, there exists a neighborhood $V_2\subset V_1$ of $x_0$ such that $|Df(x)|<\frac{\eps}{9},$ for every $x\in V_2$. 

Since $\cP(x_0)=x_0$, we can choose a smaller neighborhood $V_3$ such that $\cP(x),\cP^2(x)\in V_2$, for every $x\in V_3$. This 
ensures that $$\frac{|f(q_x)-f(x)|}{d(x,\cP(x))}<\frac{\eps}{9}\frac{d(q_x,x)}{d(x,\cP(x))}\leq\frac{\eps}{3}.$$
Similar estimates hold with $\overline{q}_x$ in place of $q_x$ and in place of $x$, respectively.  

By Inequality (\ref{e.ftemdoissinais}) we see that $f(q_x)$ and $f(\bar q_x)$ have opposite signs and thus 
$$\frac{|f(q_x)+f(\bar q_x)|}{d(x,\cP(x))}\leq\frac{|f(q_x)-f(\bar q_x)|}{d(x,\cP(x))}\leq\frac{\eps}{3}.$$ 

We deduce 
\begin{eqnarray*}
\left|\frac{2f(x)}{d(x,\cP(x))}\right|&=&\frac{|f(x)-f(q_x)+f(q_x)+f(\bar q_x)+f(x)-f(\bar q_x)|}{d(x,\cP(x))}\\
&\leq&\frac{|f(x)-f(q_x)|}{d(x,\cP(x))}+\frac{|f(\bar q_x)+(q_x)|}{d(x,\cP(x))}\frac{|f(x)-f(\bar q_x)|}{d(x,\cP(x))}\\
&\leq&\eps.
\end{eqnarray*}
This establishes that $\frac{|\mu(\cP(x))-\mu(x)|}{d(\cP(x),x)}$ is
smaller than $\eps$ for every $x\in V_3\cap\Sigma_+$ 
and completes the proof.
\end{proof}

\begin{rema}\label{r.angle} The Lemmas~\ref{l.angle} and~\ref{l.angle2} depend a priori on the choice of coordinate on $\Sigma_0$ 
since they are formulated
in terms of segments of straight line $[x,\cP(x)]$, and vectors $\cP(x)-x$. Nevertheless, the choice of coordinates on
$\Si_0$ was arbitrary (see first paragraph of Section~\ref{ss.angular}) so that 
it holds indeed for any choice of $C^1$ coordinates on $\Sigma_0$ (on a neighborhood of $x_0$ depending on the choice of the coordinates). 
\end{rema}

\paragraph{Dynamics of the first return map $\cP$ in the neighborhood of $0$}

Recall that $\Sigma_0$ is  a disc endowed with an arbitrary (but fixed) choice of coordinates.
Also, $\mu\colon\Sigma_0\to \RR$ is a $C^1$-map whose derivative do not vanish along $\col{(X,Y,U)}$ and 
$\col{(X,Y,U)}\cap \Sigma_0$ is a $C^1$-curve. 

Therefore, one can choose a $C^1$-map $\nu\colon \Sigma_0\to \RR$ so that
\begin{itemize}
 \item there is a neighborhood $O$ of $x_0$ in $\Sigma_0$ so that $(\mu,\nu)\colon O\to\RR^2$ is $C^1$ diffeomorphism,
 \item $\col{(X,Y,U)}\cap O=\nu^{-1}(\{0\})$
 \item $\nu>0$ on $\Sigma_+$
\end{itemize}
We denote by $(\mu(x),\nu(x))$ the image of $x$ by $(\mu,\nu)$. 

Notice that $(\mu(x),\nu(x))$ are local coordinates on $\Sigma_0$ in a neighborhood of $x_0$. 
Remark~\ref{r.angle} allows us to use Lemma~\ref{l.angle}
and Lemma~\ref{l.angle2} in the coordinates $(\mu,\nu)$.

 As a consequence, there exist $\varepsilon>0$ so that for any point $x\in\Sigma_+$ with 
 $(\mu(x),\nu(x))\in [-\varepsilon,\varepsilon]\times(0,\varepsilon]$, one has
 \begin{equation}\label{e.100}
 |\mu(\cP(x))-\mu(x)|<\frac1{100}|\nu(\cP(x))-\nu(x)|.
 \end{equation}

In particular, since $\cP$ has no fixed point in $\Sigma_+$ (Corollary~\ref{c.nofixed}), 
one gets that $\nu(\cP(x))-\nu(x)$ does not vanish for 
$(\mu(x),\nu(x))\in [-\varepsilon,\varepsilon]\times(0,\varepsilon]$, 
and in particular it has a constant sign. Up to change $\cP$ by its inverse 
$\cP^{-1}$ (which is equivalent to replace $Y$ by $-Y$), one may assume

$$\nu(\cP(x))-\nu(x)<0 , \mbox{ for } (\mu(x),\nu(x))\in [-\varepsilon,\varepsilon]\times(0,\varepsilon] $$

Next lemma allows us to define stable sets for the points in $\col{(X,Y,U)}$ and 
shows that every point in $\Sigma_+$ close to $x_0$ 
belongs to such a stable set.

\begin{lemm}
\label{l.stableset} 
Let $x\in\Sigma_+$ be such that 
$(\mu(x),\nu(x))\in [-\frac9{10}\varepsilon,\frac 9{10}\varepsilon]\times(0,\varepsilon]$.
Then, for any integer $n\geq 0$, $\cP^n(x)$ satisfies 
$$\left(\mu(\cP^n(x)),\nu(\cP^n(x))\right)\in [-\varepsilon,\varepsilon]\times(0,\varepsilon]$$
 
 Furthermore the sequence $\cP^n(x)$ converges to a point $x_\infty\in\col{(X,Y,U)}\cap \Sigma_0$  and we have 
 \begin{itemize}
 \item $\nu(x_\infty)=0$
  \item  $\mu(x_\infty)-\mu(x)\leq \frac{\nu(x)}{100}\leq\frac\varepsilon{100}$.
 \end{itemize}
 
The map $x\mapsto x_\infty$ is continuous. 
\end{lemm}
\begin{proof} 

Consider the trapezium $D$ (in the $(\mu,\nu)$ coordinates) 
whose vertices are 
$(-\varepsilon,0)$, $(\eps,0)$, $(-\frac{9\eps}{10},\eps)$, and  $(\frac{9\eps}{10},\eps)$. This trapezium $D$
 is contained in $[-\eps,\eps]\times [0,\eps]$ and contains 
 $[-\frac 9{10}\eps,\frac 9{10}\eps]\times [0,\eps]$.  Thus for proving the first item it is enough to 
 check that $D$ is invariant under $\cP$.  For that notice that, for any $x$ with 
 $(\mu(x),\nu(x))\in [-\eps,\eps]\times [0,\eps]$   one has that $(\mu(\cP(x)),\nu(\cP(x))$ belongs to
 the triangle $\delta(x)$  whose vertices are $(\mu(x),\nu(x))$, $(\mu(x)-\frac{\nu(x)}{100},0)$, 
 $(\mu(x)+\frac{\nu(x)}{100},0)$ (according to Equation~\ref{e.100} and the fact that $\nu(\cP(x))<\nu(x)$); one conclude by noticing that, 
 if $(\nu(x),\mu(x))$ belongs to $D$ then $\delta(x)\subset D$.

 Let us show that $\cP^n(x)$ converges.

The sequence $\nu(\cP^n(x))$ is positive and decreasing, hence converges, and 
$$\sum |\nu(\cP^{n+1}(x))-\nu(\cP^n(x))|$$ converges. 
As $|\mu(\cP^{n+1}(x))-\mu(\cP^n(x))|<\frac1{100}|\nu(\cP^{n+1}(x))-\nu(\cP^n(x))|$ one deduces that the sequence $\{\mu(\cP^n(x)\}_{n\in\NN}$ is
a Cauchy sequence, hence converges. 

The continuity of $x\mapsto x_\infty$ follows from the inequality 
$\mu(x_\infty)-\mu(x)\leq \frac{\nu(x)}{100}$ applied to $\cP^n(x)$ with $n$ large, 
so that $\nu(\cP^n(x))$
is very small, and from the continuity of $x\mapsto\cP^n(x)$. 
\end{proof}

For any point $y$ with $(\mu(y),\nu(y))\in[-\frac{8}{10}\varepsilon,\frac{8}{10}\varepsilon]\times \{0\}$ the
\emph{stable set of $y$}, which we 
denote by $S(y)$, is the the union of $\{y\}$ with the set of points $x\in\Sigma_+$ with $(\mu(x),\nu(x))\in [-\frac9{10}\varepsilon,\frac 9{10}\varepsilon]\times(0,\varepsilon]$ so that 
$x_\infty=y$. The continuity of the map $x\mapsto x_\infty$ implies the following remark:
\begin{rema}
 For any point $y$ with $(\mu(y),\nu(y))\in[-\frac{8}{10}\varepsilon,\frac{8}{10}\varepsilon]\times \{0\}$, $S(y)$ is a compact set which has a 
 non-empty intersection with the horizontal lines $\{x, \nu(x)=t\}$ for every $t\in(0,\varepsilon]$. 
\end{rema}

If $E$ is a subset of $\col{(X,Y,U)}\cap\Sigma_0$ so that $\mu(y)\in[-\frac{8}{10}\varepsilon,\frac{8}{10}\varepsilon]$ for $y\in E$ one denotes 
$$S(E)=\bigcup_{y\in E}(S(y)).$$

\begin{lemm}Let $I\subset \col{(X,Y,U)}$ be the open interval 
$(\mu(x),\nu(x))\in(-\frac{1}{2}\varepsilon,\frac{1}{2}\varepsilon)\times \{0\}$. 
Consider the quotient space $\Gamma$ of $S(I)\setminus I$ by the dynamics. 
Then $\Gamma$ is a $C^1$-connected surface diffeomorphic to a cylinder 
$\RR/\ZZ\times \RR$.   
\end{lemm}
\begin{proof} Consider the compact triangle whose end points are 
$(-\frac{\eps}{2},0)$, $(0,\frac{\eps}{10})$ and $(+\frac{\eps}{2},0)$. Let $\bar \Delta$ be its preimage by $(\mu,\nu)$. 

$\bar \Delta$ is a triangle with one side on $\col(X,Y,U)$. Let $\Delta=\bar\Delta\setminus\col(X,Y,U)$.
We denote by $\partial \Delta$ the union of the two other sides. 

As the vectors directing the two other sides have a first coordinated larger
than the second (in other words, they are more horizontal than vertical) and as the vectors $[x,\cP(x)]$ are almost vertical, one deduces that $\Delta$ is a trapping region for $\cP$: 
$$x\in\Delta \Longrightarrow\cP(x)\in \Delta.$$

Now $$\cP(\partial\Delta)$$ is a curve contained in the interior of $\Delta$ and joining  the vertex $(-\frac{\eps}{2},0)$ to the vertex 
$(+\frac{\eps}{2},0)$.  Thus $\partial\Delta$ and $\cP(\partial\Delta)$ bound a strip 
homeomorphic to $[0,1]\times \RR$ in $\Delta$. 

Let $\Gamma$ be the cylinder obtained from this strip by 
gluing $\partial\Delta$ with $\cP(\partial\Delta)$ along $\cP$. 

It remains to check that $\Gamma$ is the  quotient space of $S(I)\setminus I$ by $\cP$. 
For that, one just remark that the orbit of every point in $S(I)\setminus I$ has a unique point in the strip, unless in the case where the 
orbits meets $\partial \Delta$: in that case the orbits meets the strip twice, the first time  on $\partial \Delta$, the second on 
$\cP(\partial\Delta$).

Note that $\Gamma$ is the quotient of an open region in a smooth surface by a diffeomorphism (and the action is proper and free) so that
$\Gamma$ is not only homeomorphic to a cylinder but diffeomorphic to a cylinder. 
\end{proof}

Let us end this section by stating important straighforward consequences of the 
invariance of the vector field $N$ under $\cP$. 
\begin{rema}
\begin{itemize}
\item For every $y\in\col(X,Y,U)$ with $\mu(y)\in[-\frac{8}{10}\varepsilon,\frac{8}{10}\varepsilon]$, 
 the stable set $S(y)$ is invariant under the flow of $N$.
\item The vector field $N$ induces a vector field, denoted by $N_\Gamma$, on the quotient space $\Gamma$. As $\Gamma$ is a $C^1$ surface, 
$N_\Gamma$ is only $C^0$. However, it defines a flow on $\Gamma$ which is the quotient by $\cP$ of the flow of $N$. 
\item The continuous map $x\mapsto x_\infty$ is invariant under $\cP$ and therefore induces on $\Gamma$ a continuous map
$\Gamma\to (-\varepsilon/2,\varepsilon/2)$, and $x_\infty$ tends to $-\varepsilon/2$ when $x$ tends to one end of the cyclinder $\Gamma$ and to $\varepsilon/2$ when $x$ tends to 
the other end. This implies that, for every $t\in(-\varepsilon/2,\varepsilon/2)$ the set of points $x\in\Gamma$ for which $x_\infty=t$ is compact. 
Recall that this set is precisely the projection on $\Gamma$ of $S(y)$ where $y\in\col(X,Y,U)$ satisfies $(\mu(y),\nu(y))=(t,0)$. 
\item The vector field $N_\Gamma$ on $\Gamma$ leaves invariant the levels of the map $x\mapsto x_\infty$.  As a consequence, every orbit of $N_\Gamma$
is bounded in $\Gamma$.
\end{itemize}
\end{rema}

One deduces

\begin{lemm}\label{l.invariant}For every $y\in\col(X,Y,U)$ with $\mu(y)\in[-\frac{8}{10}\varepsilon,\frac{8}{10}\varepsilon]$, the stable set 
$S(y)\setminus \{y\}$ contains an orbit of $N$ which is invariant under $\cP$.
 In particular, there exists $x$ in $\Delta\cap S(x_0)$ whose orbit by $N$ is invariant under $\cP$. 
\end{lemm}
\begin{proof}
A Poincar\'e Bendixson argument implies that, for every flow on the cylinder $\RR/\ZZ\times \RR$ without fixed points, for every bounded orbit
the $\omega$-limit set is a periodic orbit. Furthermore this periodic orbit is not homotopic to a point 
(otherwise it bounds a disc containing a zero). 

Since $N_{\Gamma}$ has no zeros, one just applies this argument to the flow of it restricted to the a level of the map $x\mapsto x_\infty$. The level contains a periodic orbit 
which is not homotopic to $0$, hence corresponds to an orbit of $N$ joining a point in $S(y)$ to is image under $\cP$.  
This orbit of $N$ is invariant under $\cP$, concluding.
\end{proof}

\paragraph{End of the proof of Theorem \ref{teoprincipal}: 
the vector field $N$ does not rotate along a $\cP$-invariant orbit of $N$}
 
 From now on, $(U,X,Y,\Si,\cB)$ is a prepared counter example to Theorem~\ref{teoprincipal}  
 with $\ell^+(X,Y)\neq 0$.
According to Proposition~\ref{p.derivada} the derivative $D\cP(x)$ is the 
identity map for every $x\in\col{(X,Y,U)}\cap\Si_0$.

According to Lemma~\ref{l.invariant}, there is a point $x$ in the stable set $S(x_0)$ whose $N$-orbit 
is invariant under $\cP$. 

\begin{lemm}
\label{l.naogiradois}
If $y\in[\cP^n(x),\cP^{n+1}(y)]$ then the angular variation of the vector $N(y)$ tends to $0$ when $n\to +\infty$.
\end{lemm}

Before proving  Lemma~\ref{l.naogiradois} let us conclude the proof of Theorem~\ref{teoprincipal}.

\begin{proof}[Proof of Theorem~\ref{teoprincipal}] Lemma~\ref{l.naogiradois}  is 
in contradiction with Corollary~\ref{c.gira}, 
which asserts that the angular variation 
of the vector $N$ along any segment $[z,\cP^2(z)]$ for $z\in\Sigma_+$ close enough to $x_0$ is larger than $2\pi$. This contradiction ends the proof of 
Theorem~\ref{teoprincipal}. 
\end{proof}

It remains to prove Lemma~\ref{l.naogiradois}. First, notice that the angular variation of $N$ along a segment of curve is invariant 
under homotopies of the curve preserving the ends points. 
Therefore Lemma~\ref{l.naogiradois} is a straighforward consequence of next lemma: 

\begin{lemm}
\label{l.naogiradois2} 
The angular variation of the vector $N(y)$ 
along the $N$-orbit segment joining $\cP^n(x)$ to $\cP^{n+1}(x)$ 
tends to $0$ when $n\to +\infty$.
\end{lemm}

As $N$ is (by definition) tangent to the $N$-orbit segment joining $\cP^n(x)$ to $\cP^{n+1}(x)$, 
its angular variation is equal 
to the angular variation of the unit tangent vector to this orbit segment. 

\paragraph{Proof of Lemma~\ref{l.naogiradois2}: the tangent vector to a 
$\cP$-invariant embedded curve do not rotate}

\begin{rema}
For $n$ large enough the point $\cP^n(x)$ belongs to the region $\Delta$ defined in the previous section, and whose quotient by the dynamics 
$\cP$ is the cyclinder $\Gamma$.  Then,

\begin{itemize}
\item any continuous curve $\gamma$ in $\Delta$ joining $\cP^n(x)$ to $\cP^{n+1}(x)$ induces on 
$\Gamma$ a closed curve, 
homotopic to the curve induced by the $N$-orbit segment joining $P^n(x)$ to $\cP^{n+1}(x)$. 
\item The curve induced by $\gamma$ is a simple curve if and only if $\gamma$ is simple and disjoint 
from $\cP^i(\gamma)$ for any $i>0$. 
\item If the curve $\gamma$ is of class $C^1$, 
the projection will be of class $C^1$ if and only if the image by $\cP$ of the unit  vector tangent 
to $\gamma$ at $\cP^n(x)$ 
is tangent to $\gamma$ (at $\cP^{n+1}(x)$). 

\item on the cylinder, any two $C^1$-embbedding $\sigma_1,\sigma_2$ of the circle, so that $\sigma_1(0)=\sigma_2(0)$ are isotopic through
$C^1$-embeddings $\sigma_t$ with $\sigma_t(0)=\sigma_i(0)$.

\end{itemize}
 
\end{rema}

We consider $\cI(\Delta,\cP)$ as being the set of $C^1$-immersed segment $I$ in $\Delta\setminus\partial\De$, so that:
\begin{itemize}
 \item if $y,z$ are the initial and end points of $I$ then $z=\cP(y)$
 \item if $u$ is a vector tangent to $I$  at $y$ then $\cP_*(u)$ is tangent to $I$ (and with the same 
 orientation.
\end{itemize}
In other words, $I\in \cI(\Delta,\cP)$ if the projection of $I$ on $\Gamma$ is a $C^1$ immersion of the 
circle, 
generating the fundamental group of the cylinder. We endow $\cI(\Delta,\cP)$ with the $C^1$-topology. 

We denote by $Var(I)\in \RR$ the angular variation of the unit tangent vector to $I$ along $I$. 
In other words, for $I\colon [0,1]\to \Delta$,  consider the unit vector 
$$\mathring{I}(t)=\frac{dI(t)/dt}{\|dI(t)/dt\|}\in \SS^1\simeq \RR/2\pi\ZZ.$$  
One can lift $\mathring{I}$ is a continuous map $\dot{I}\colon [0,1]\to \RR$.  Then 
$$Var(I)=\dot{I}(1)-\dot{I}(0),$$
this difference does not depend on the lift. 

The map $Var\colon \cI(\Delta,\cP)\to \RR$ is continuous.  Let $\widetilde{Var}(I)\in \RR/2\pi\ZZ$ be 
the projection of $Var(I)$.  In other words, $\widetilde{Var}(I)$ is the angular variation modulo $2\pi$.

\begin{rema}\label{r.var}Let $I_n\in\cI(\Delta,\cP)$ be a sequence of immersed segments such that $I_n(0)$ tends to 
$x_0\in\Sigma$. Then $\widetilde{Var}(I_n)$ tends to $0$. 

Indeed, since $D\cP(I_n(0))$ tends to the identity map, the angle between  
the tangent vectors to $I_n$ at $I_n(1)$ and $I_n(0)$ tends to $0$.
\end{rema}

As a consequence of  Remark~\ref{r.var}, we get the following lemma:
\begin{lemm}
There is a neighborhood $O$ of $x_0$ in $\Sigma$ so that  to any 
$I\in\cI(\Delta,\cP)$ with $I(0)\in O$  
there is a (unique) integer $[var](I)\in\ZZ$ so 
that  $$Var(I)-2\pi [var](I)\in\left[-\frac 1{100},\frac 1{100}\right].$$ 

Furthermore, the map $I\mapsto[var](I)$ is locally constant in $\cI(\Delta,\cP)$, hence constant under 
homotopies in $\cI(\Delta,\cP)$ keeping the initial point in $O$. 
\end{lemm}

As a consequence we get
\begin{lemm}\label{l.var}
If $I$ and $J$ are segments in $\cI(\Delta,\cP)$ with the same intial point in $O$ and whose 
projections on the cyclinder $\Gamma$ are simple closed curves, then 
$$[var](I)=[var](J).$$
\end{lemm}
\begin{proof}
Since the projection of $I$ and $J$ are simple curves which are not homotopic to a point 
in the cylinder $\Gamma$, the projections of $I$ and $J$
are isotopic on $\Gamma$ by an isotopy keeping the initial point.  
One deduces that $I$ and $J$ are homotopic through elements $I_t\in \cI(\Delta,\cP)$ with the same
initial point. Indeed, the isotopy on $\Ga$ between the projection of $I$ and $J$ can 
be lifted to the universal cover of $\Ga$.  This universal cover is diffeomorphic to a plane $\RR^2$, in which
$\De\setminus\partial\Delta$ is an half plane (bounded by two half lines). There is a diffeomorphism of $\RR^2$
to $\De\setminus\partial\Delta$ which is the indentity on $I\cup J$. The image of the lifted isotopy induces 
the announced isotopy through elements in $\cI(\De,\cP)$. Now, as $[var](I_t)$ is independent of $t$,
one concludes that $[var](I)=[var](J)$.
\end{proof}

Now Lemma~\ref{l.naogiradois2} is a consequence of Lemma~\ref{l.var} and of the folowing lemma:

\begin{lemm}\label{l.straight}
 For any $n>0$ large enough there is a curve $I_n\in \cI(\Delta,\cP)$ whose initial point is $\cP^n(x)$ 
 and such that:
 \begin{itemize}
 \item the projection of $I_n$ on $\Gamma$ is a simple curve
  \item $[var](I_n)=0$.
 \end{itemize}
\end{lemm}

\begin{proof}[End of the proof of Lemma~\ref{l.naogiradois2}] The $N$-orbit segment $J_n$ joining 
$\cP^n(x)$ to $\cP^{n+1}(x)$ belongs to $\cI(\Delta,\cP)$ and its projection on $\Gamma$ is a simple curve.  
Hence  Lemma~\ref{l.var} asserts that, for $n$ large enough,  $[var](J_n)=[var](I_n)=0$ where $I_n$ 
is given by Lemma~\ref{l.straight}.

Now, when $n$ tends to infinity, $Var(J_n)-2\pi[var](J_n)$ tends to $0$ (according to Remark~\ref{r.var}), 
that is, $Var(J_n)$ tends to $0$.  This is precisely
the statement of Lemma~\ref{l.naogiradois2}.
\end{proof}

\begin{proof}[Proof of Lemma~\ref{l.straight}] As $n$ tends to $+\infty$ the derivative of $\cP$ at $\cP^n(x)$ tends to the identity map.  Thus, 
the segment
$[\cP^n(x),\cP^{n+1}(x)]$ may fail to belong to $\cI(\Delta,\cP)$ only by a very small angle between $v_n$ and $D\cP(v_n)$, where $v_n$ is the
unit 
vector directing $[\cP^n(x),\cP^{n+1}(x)]$.  Therefore, one easily builds an arc $I_n$ joining $\cP^n(x)$ 
to $ \cP^{n+1}(x)$, whose derivative at
$\cP^n(x)$ is $v_n$ and its derivative at $\cP^{n+1}(x)$ is $D\cP(v_n)$ and whose derivative at any point of $I_n$ belongs to an arbitrarily small neighborhood of
$v_n$. In particular $I_n\in\cI(\Delta,\cP)$ and $[var](I_n)=0$ for $n$ large. In order to complete the proof, it remains to show that
\begin{clai} For $n$ large enough
$I_n$ projects on $\Gamma$ as a simple curve. 
\end{clai}
\begin{proof} We need to prove that for $n$ large enough and for any $i>1$,
$I_n\cap \cP^i(I_n)=\emptyset$ and $I_n\cap \cP(I_n)$ is a singleton (the endpoint of 
$I_n$ which is the image of its initial point). 

Indeed, it is enough to prove that, for any $y\in I_n$ different from the initial point, 
$$
\nu(\cP(y))<\inf_{z\in I_n} \nu(z).
$$
As the action of $\cP$ consists in lowing down the value of $\nu$, the further iterates cannot cross $I_n$. 

For proving that, notice that the vectors tangent to $I_n$ are very close to $v_n$ 
which is uniformly (in $n$ large) transverse 
to the levels of $\nu$.
As $D\cP(y)$ tends to the identity map when $y$ tends to $0$, for $n$ large, 
the vectors tangent to $\cP(I_n)$ are also transverse to the levels of $\nu$. 
Hence 
$\cP(I_n)$ is a segment starting at the end point of $I_n$ 
(which realizes the infimum of $\nu$ on $I_n$) and $\nu$ is strictly decreasing along 
$\cP(I_n)$, concluding.
\end{proof}
This ends the proof of Lemma~\ref{l.straight} (and so of Theorem~\ref{teoprincipal}).
\end{proof}

\end{document}